\documentclass[11pt,reqno]{amsart}
\usepackage{tikz}
\usepackage{subfig}
\usepackage{epstopdf}
\usepackage{epsfig}
\usepackage{math}
\graphicspath{{./figures/}}

\usepackage{tikz}
\usetikzlibrary{shapes,arrows}

\tikzstyle{decision} = [diamond, draw, fill=blue!20, 
    text width=4.5em, text badly centered, node distance=3cm, inner sep=0pt]
\tikzstyle{block} = [rectangle, draw, fill=blue!20, 
    text width=5em, text centered, rounded corners, minimum height=4em]
\tikzstyle{line} = [draw, -latex']
\tikzstyle{cloud} = [draw, ellipse,fill=red!20, node distance=3cm,
    minimum height=2em]

\usetikzlibrary{positioning}
\tikzset{main node/.style={circle,fill=blue!20,draw,minimum size=1cm,inner sep=0pt},  }

\begin{document}
\title[Semi-discretization schemes]{Entropy dissipation semi-discretization schemes for Fokker-Planck equations}
\author[Chow]{Shui-Nee Chow}
\address{School of Mathematics, Georgia Institute of Technology,
Atlanta, GA 30332 U.S.A.}
\email{chow@math.gatech.edu}
\author[Dieci]{Luca Dieci}
\email{dieci@math.gatech.edu}
\author[Li]{Wuchen Li}
\email{wcli@gatech.edu}
\author[Zhou]{Haomin Zhou}
\email{haomin.zhou@math.gatech.edu}

\thanks{This work is partially supported by NSF 
Awards DMS\textendash{}1042998, DMS\textendash{}1419027,
and ONR Award N000141310408.}



\begin{abstract}
We propose a new semi-discretization scheme to approximate nonlinear Fokker-Planck equations,
by exploiting the gradient flow structures with respect to the 2-Wasserstein metric in the space of probability densities.
We discretize the underlying state by a finite graph and define a discrete 2-Wasserstein metric in the discrete probability space. 
Based on such metric, we introduce a gradient flow of the discrete free energy as semi discretization scheme. We prove that the scheme maintains dissipativity of the 
free energy and converges to a discrete Gibbs measure at exponential dissipation rate. 
We exhibit these properties on several numerical examples. 
\end{abstract}

\maketitle

\section{Introduction}

In this paper we introduce and study semi-discretization schemes for certain types of partial 
differential equations (PDEs) \cite{JKO}, which are gradient flows from the viewpoint 
of optimal transportation theory 
\cite{ambrosio2003lecture, am2006, bb, carrillo2003kinetic, O, OV, vil2003, vil2008}.

Consider a nonlinear Fokker-Planck equation \cite{benedetto1998non, carrillo2003kinetic} 
\begin{equation}\label{PDE}
\frac{\partial \rho}{\partial t}=\nabla\cdot[\rho \nabla (V(x)+ 
\int_{\mathbb{R}^d}W(x,y)\rho(t,y)dy)]+\beta \Delta \rho.
\end{equation}
The unknown $\rho(t, \cdot)$ is a probability density function supported on $\mathbb{R}^d$, the functions
$V: \mathbb{R}^d\rightarrow \mathbb{R}$, and $W: \mathbb{R}^d\times \mathbb{R}^d\rightarrow \mathbb{R}$
are smooth and further $W(x,y)=W(y,x)$ for any $x, y\in \mathbb{R}^d$.

To \eqref{PDE} is associated a functional $\mathcal{F}: \mathcal{P}(\mathbb{R}^d)\rightarrow \mathbb{R}$, 
called free energy 
\begin{equation}\label{energy}
\mathcal{ F}(\rho)=\int_{\mathbb{R}^d}V(x)\rho(x)dx+\frac{1}{2}
\int_{\mathbb{R}^d\times \mathbb{R}^d}W(x,y)\rho(x)\rho(y)dxdy+\beta\int_{\mathbb{R}^d}\rho(x)\log\rho(x)dx,
\end{equation}
which is a summation of linear potential energy, interaction energy and 
linear entropy, from left to right. 
It is known that the free energy \eqref{energy} is a Lyapunov function for \eqref{PDE}:
\begin{equation*}
\frac{d}{dt}\mathcal{F}(\rho(t,\cdot))=-\int_{\mathbb{R}^d} (\nabla F(x,\rho))^2\rho(t,x) dx\leq 0,
\end{equation*} 
where $F(x,\rho):=\frac{\delta}{\delta\rho(x)}\mathcal{F}(\rho)(x)$, and $\frac{\delta}{\delta\rho(x)}$ is the $L^2$ first variation.
Under suitable conditions on $V$ and $W$, the solution $\rho(t,\cdot)$ of \eqref{PDE} 
converges to an equilibrium $\rho^*(x)$ named Gibbs measure, where 
\begin{equation*}
\rho^*(x)=\frac{1}{K}e^{-\frac{V(x)+\int_{\mathbb{R}^d} W(x,y)\rho^*(y)dy}{\beta}},\quad 
\textrm{where}\quad K=\int_{\mathbb{R}^d} e^{-\frac{V(x)+\int_{\mathbb{R}^d} W(x,y)\rho^*(y)dy}{\beta}} dx.
\end{equation*}
Recent work on optimal transport treats the probability set $\mathcal{P}(\mathbb{R}^d)$
as a ``Remannian manifold'' equipped with the 2-Wasserstein metric. 
From this viewpoint, \eqref{PDE} is a gradient flow of the
free energy $\mathcal{F}(\rho)$ on $\mathcal{P}(\mathbb{R}^d)$, 
see \cite{am2006, vil2003, vil2008}. Furthermore, requiring $W(x,y)=W(|x-y|)$ with suitable conditions, 
Carrillo, McCann and Villani show that $\rho(t,\cdot)$ converges to a Gibbs measure
with exponential rate, see \cite{carrillo2003kinetic}.

In this paper, we consider a similar matter in the discrete setting. In other words, we shall derive a semi-discretization scheme for \eqref{PDE} (continuous in time and discrete in spatial space), 
which also has a gradient flow structure with respect to a discrete 2-Wasserstein metric in the discrete probability space. It is worth mentioning that the underlying space \eqref{PDE} can be a variety other 
than $\mathbb{R}^d$.  For instance, the domain can be a bounded open set,
with a zero-flux conditions or periodic conditions. In this paper, we apply the setting of finite graph to consider all these cases. 

Consider a graph $G=(V,E)$ to discretize the spatial domain, where 
$V$ is the vertex set  
\begin{equation*}
V=\{1,2,\cdots, n\},
\end{equation*} 
and $E$ is the edge set.
The adjacency set of the vertex $i\in V$ is denoted by
\begin{equation*}
N(i)=\{j\in V\mid (i,j)\in E\}.
\end{equation*} 
Here $i\in V$ represents a point in $\mathbb{R}^d$, and $(i,j)$ is shorthand 
for an edge connecting $i$ and $j$. For concreteness, we can think of $G$ as a lattice corresponding to a uniform
discretization of the domain with spacing $\Delta x$. 

Consider a discrete probability set
supported on all vertices:
\begin{equation*}
\mathcal{P}(G)=\{\rho=(\rho_i)_{i=1}^n\in \mathbb{R}^{n}\mid \sum_{i=1}^n\rho_i=1,~\rho_i\geq 0,~i\in V \}.
\end{equation*}
Moreover, we consider a discrete free energy of $\mathcal{ F}(\rho)$, as an analog of \eqref{energy}
\begin{equation*}
\mathcal{F}(\rho)=\sum_{i=1}^nv_i\rho_i+\frac{1}{2}\sum_{i=1}^n\sum_{j=1}^n
w_{ij}\rho_i\rho_j+\beta\sum_{i=1}^n\rho_i\log\rho_i,
\end{equation*}
where $(v_i)_{i=1}^n=(V(i))_{i=1}^n$ and $(w_{ij})_{1\leq i, j\leq n}=(W(i, j))_{1\leq i, j\leq n}$ 
are a fixed vector, and a fixed symmetric matrix, respectively. By this setting, we will equip $\mathcal{P}(G)$ with a ``discrete'' 2-Wasserstein metric, then derive and analyze the gradient flow of discrete free energy under this metric. 

Delaying the derivation details until section \ref{derivation}, we show the 
semi-discretization directly.  We propose to take 
\begin{equation}\label{a1}
\frac{d\rho_i}{dt}=\frac{1}{\Delta x^2}\{\sum_{j\in N(i)} \rho_j( F_j(\rho)- F_i(\rho))_+
-\sum_{j\in N(i)}\rho_{i}( F_i(\rho)-F_j(\rho))_+\},
\end{equation}
where $i\in V$, $(\cdot)_+=\max\{\cdot, 0\}$ and
\begin{equation*}
F_i(\rho)=\frac{\partial}{\partial \rho_i}\mathcal{F}(\rho), \quad \textrm{for any $i\in V$}\ .
\end{equation*}
Besides showing that \eqref{a1} is a well defined ordinary differential equation (ODE), 
we demonstrate that \eqref{a1} has a gradient flow structure. Firstly, the free energy is
a Lyapunov function of \eqref{a1}:
\begin{equation*}
\frac{d}{dt}\mathcal{F}(\rho(t))=-\sum_{(i,j)\in E}[(\frac{F_i(\rho)-F_j(\rho)}{\Delta x})_+]^2\rho_i\leq 0.
\end{equation*}
Then, if $\rho(t)$ converges to an equilibrium $\rho^{\infty}$, then we will show that
such equilibrium is a discrete Gibbs measure 
\begin{equation*}
\rho_i^{\infty}=\frac{1}{K}e^{-\frac{v_i+\sum_{j=1}^nw_{ij}\rho_j^{\infty}}{\beta}},\quad K=
\sum_{i=1}^ne^{-\frac{v_i+\sum_{j=1}^nw_{ij}\rho_j^{\infty}}{\beta}}.
\end{equation*}
Furthermore, if $\rho^{\infty}$ is a strictly local minimizer of the free energy,
and $\rho(t)$ is in its basin of attraction for the gradient dynamics,
then we will show that the convergence speed
is exponential:
\begin{equation*}
\mathcal{F}(\rho(t))-\mathcal{F}(\rho^{\infty})\leq e^{-Ct}(\mathcal{F}(\rho^0)-\mathcal{F}(\rho^{\infty})),
\end{equation*}
where $C$ is a positive constant. 
In fact, we will say more about this convergence. 
We will give an explicit formula for the asymptotic convergence rate, which mimics 
the role of the Hessian of the free energy at the Gibbs measure w.r.t. the discrete 2-Wasserstein metric. Finally, we will show that \eqref{a1} is a consistent scheme for the PDE \eqref{PDE}, and further derive a general consistent scheme for general drift diffusion systems, not necessarily gradient flows. 
 
The semi-discretization scheme in this paper is largely inspired by \cite{chow2012, li_thesis},
the upwind scheme of \cite{upwind}, and optimal transport theory \cite{vil2008}. 
In addition, the convergence result is influenced by the work of Carrillo, McCann and Villani,  
\cite{carrillo2003kinetic}. Our method can be viewed as a discrete 
entropy dissipation method \cite{JJM, EP}, with a dynamical twist. 
 
In the literature, people have studied 2-Wasserstein metric and Fokker-Planck equations
in discrete settings for a long time \cite{CC,  c2015, c2016, EM2, EM1,  M}. 
Maas \cite{EM1} and Mielke \cite{M} introduce a different discrete 2-Wasserstein metric. 
Based on such metric, they analyze the convergence rate of some schemes for one-dimensional
linear Fokker-Planck equations. Our scheme shows exponential convergence for all 
linear and nonlinear cases. Carrillo, Chertock, Huang, Wolansky \cite{c2015, c2016} 
have recently designed several algorithms based on entropy dissipation viewpoint.
Particularly, the first order scheme designed in \cite{c2015}
shares some similarities with \eqref{a1} for a lattice graph. 
However, we advocate designing semi discretization schemes by using directly the viewpoint of discrete Wasserstein metric. We believe that the metric would be useful for deriving various time discretization for semi discretization scheme in the light of \cite{JKO}. In addition, the gradient flow of entropy with this metric suggests an interesting nonlinear discretization of Laplacian operator. This effect introduces many dynamical properties of the semi-discretization scheme, such as exponential convergence. 

This paper is arranged as follows. In section \ref{derivation}, we derive \eqref{a1}
based on a discrete 2-Wasserstein metric.  With respect to this metric, \eqref{a1}'s gradient
flow properties are given. 
In section \ref{convergence}, we show that the solution of \eqref{a1} converges to a
discrete Gibbs measure exponentially fast. 
Numerical analysis and several experiments on \eqref{a1} are discussed in sections
\ref{na} and \ref{numerical}.   

\section{Semi-discretization scheme}\label{derivation}
In this section, we show that \eqref{a1} is a gradient flow for the discrete free energy 
on the probability set $\mathcal{P}(G)$.
First, we define a discrete 2-Wasserstein metric on $\mathcal{P}(G)$. Second, based on such metric, 
we derive \eqref{a1} as a gradient flow of the discrete free energy.
Third, we show that \eqref{a1} is a well defined ODE with gradient flow structure.  
\subsection{Discrete 2-Wasserstein metric}
The 2-Wasserstein metric (Benamou-Brenier formula, \cite{bb}) is a metric defined 
on a probability set supported on $\mathbb{R}^d$: 
\begin{equation*}
W_2(\rho^0, \rho^1)^2=\inf_{\Phi}\{{\int_0^1 (\nabla\Phi, \nabla \Phi)_\rho dt}~:
~\frac{\partial \rho}{\partial t}+\nabla \cdot (\rho\nabla \Phi )=0, ~\rho(0)=\rho^0,~\rho(1)=\rho^1 \},
\end{equation*}
where $(\cdot, \cdot)_\rho$ represents an inner product on the probability set:
\begin{equation*}
(\nabla \Phi, \nabla\Phi)_\rho=\int_{\mathbb{R}^d}(\nabla \Phi(t,x))^2\rho(t,x)dx,
\end{equation*}
and the infimum is taken among the 
potential functions $\Phi(t,x)\in \mathbb{R}^d$.

We give a similar metric definition on a discrete setting, which is a finite graph $G=(V, E)$. 
Consider a probability set supported on $V$ with all positive measures:
\begin{equation*}
\mathcal{P}_o(G)=\{\rho=(\rho_i)_{i=1}^n \mid \sum_{i=1}^n \rho_i=1,\quad \rho_i>0,\quad \textrm{for any}~i\in V\}.
\end{equation*}
We use three steps to define the metric on $\mathcal{P}_o(G)$. 
Firstly, we define a potential vector field on graph 
\begin{equation*}
\nabla_G\Phi:=(\frac{1}{\Delta x}(\Phi_i-\Phi_j))_{(i,j)\in E},
\end{equation*}
with the potential function $\Phi:=(\Phi_i)_{i=1}^n$. Secondly, we introduce the discrete analog of $\nabla \cdot (\rho \nabla \Phi)$ by:
\begin{equation*}
\textrm{div}_G(\rho \nabla_G\Phi):=
\bigl(-\frac{1}{\Delta x^2}\sum_{j\in N(i)}(\Phi_i-\Phi_j)g_{ij}(\rho)\bigr)_{i=1}^n,
\end{equation*}
where \begin{equation}\label{gij}
g_{ij}(\rho):=
\begin{cases}
\rho_i\quad &\textrm{if} ~F_i(\rho)>F_j(\rho),~j\in N(i),\\
\rho_j\quad &\textrm{if}~F_i(\rho)<F_j(\rho), ~j\in N(i) ,\\
\frac{\rho_i+\rho_j}{2}\quad &\textrm{if}~F_i(\rho)=F_j(\rho), ~j\in N(i),
\end{cases}
\end{equation}
and $F_i(\rho):=\frac{\partial}{\partial \rho_i}\mathcal{F}(\rho)$. It is worth mentioning that $g_{ij}$ defined in $\eqref{gij}$ has multiple choices, such as $g_{ij}=\frac{\rho_i+\rho_j}{2}$ in \cite{li2017}. 
Lastly, we construct an inner product on $\mathcal{P}_o(G)$:
\begin{equation*}
(\nabla_G\Phi,\nabla_G \Phi)_\rho:=\frac{1}{2\Delta x^2}\sum_{(i,j)\in E}(\Phi_i-\Phi_j)^2g_{ij}(\rho),
\end{equation*}
where $\frac{1}{2}$ is due to the fact that every edge in $G$ is counted twice, i.e. $(i,j)$, $(j,i) \in E$. 

We are now ready to introduce a discrete 2-Wasserstein metric on $\mathcal{P}_o(G)$. 
\begin{definition}\label{def_metric}
For any $\rho^0$, $\rho^1\in \mathcal{P}_o(G)$, define $W_2\colon\mathcal{P}_o(G)\times\mathcal{P}_o(G)\rightarrow\mathbb{R}:$
\begin{equation*}
\left(W_{2}(\rho^0,\rho^1)\right)^2:=\inf_{\Phi}~\{\int_0^1(\nabla_G \Phi, \nabla_G \Phi)_\rho dt~:~
\frac{d\rho}{dt}+\textrm{div}_G(\rho \nabla_G\Phi)=0,\quad\rho(0)=\rho^0,\quad \rho(1)=\rho^1\},
\end{equation*}
where the infimum is taken over all $\Phi$ for which $\rho$ is a continuously differentiable curve
$\rho:[0,1]\rightarrow \mathcal{P}_o(G)$.
\end{definition}

We justify that $W_2$ in Definition \ref{def_metric} is a well defined metric. We endow $\mathcal{P}_o(G)$ with an inner product on its tangent space 
\begin{equation*}
T_\rho\mathcal{P}_o(G)=\{(\sigma_i)_{i=1}^n\in \mathbb{R}^n\mid \sum_{i=1}^n\sigma_i=0\}.
\end{equation*}
Consider the equivalence relation ``$\sim$'' in $\mathbb{R}^n$ which stands for ``modulo additive constants,''
so that the quotient space is  
\begin{equation*}
\mathbb{R}^n / \sim=\{[\Phi]\mid (\Phi_i)_{i=1}^n\in \mathbb{R}^n\}, \quad
\textrm{where}\quad [\Phi]=\{(\Phi_1+c,\cdots, \Phi_n+c)\mid c\in\mathbb{R}^1\}.
\end{equation*}
We introduce an identification map \begin{equation*}
\tau: ~\mathbb{R}^n / \sim\rightarrow T_\rho\mathcal{P}_o(G),\quad\quad 
\tau([\Phi]):=(\sum_{j\in N(i)}\frac{1}{\Delta x^2}(\Phi_i-\Phi_j)g_{ij}(\rho))_{i=1}^n.
\end{equation*} 
\begin{lemma}\label{l1}
The map $\tau:~ \mathbb{R}^n / \sim\rightarrow T_\rho\mathcal{P}_o(G)$ is a well defined map, 
linear, and one to one. 
\end{lemma}
\begin{proof}
First, we show that $\tau$ is well defined.   We denote 
\begin{equation*}
\sigma_i=\frac{1}{\Delta x^2}\sum_{j\in N(i)}(\Phi_i-\Phi_j)g_{ij}(\rho).
\end{equation*}
Our task is equivalent to show $\sum_{i=1}^n\sigma_i=0$. Indeed, 
\begin{equation*}
\begin{split}
\sum_{i=1}^n\sigma_i=&\frac{1}{\Delta x^2}\{\sum_{i=1}^n\sum_{j\in N(i)}(\Phi_i-\Phi_j)g_{ij}(\rho)\}\\
=&\frac{1}{\Delta x^2}\{\sum_{(i,j)\in E}\Phi_ig_{ij}(\rho)-\sum_{(i,j)\in E}\Phi_jg_{ij}(\rho)\}\\
&\textrm{Relabel $i$ and $j$ on the first formula}\\
=&\frac{1}{\Delta x^2}\{\sum_{(i,j)\in E}\Phi_jg_{ji}(\rho)-\sum_{(i,j)\in E}\Phi_jg_{ij}(\rho)\}=0.
\end{split}
\end{equation*}
Hence, the map $\tau$ is a well-defined linear map. 

Next, we show $\tau$ is one to one. 
Since $T_\rho\mathcal{P}_o(G)$ and $\mathbb{R}^n / \sim$ are $(n-1)$ dimensional linear spaces, 
we only need to prove $\tau$ is injective. I.e., if 
\begin{equation*}
\sigma_i=\frac{1}{\Delta x^2}\sum_{j\in N(i)}g_{ij}(\rho)(\Phi_i-\Phi_j)=0, \quad 
\textrm{for any $i\in V$},
\end{equation*}
then $[\Phi]=0$, meaning that $\Phi_1=\Phi_2=\cdots=\Phi_n$.

Assume this is not true. Let $c=\max_{i\in V}\Phi_i$. Since the graph $G$ is connected, 
there exists an edge $(k, l)\in E$, such that $\Phi_l=c$ and $\Phi_k<c$.  But, since $\sigma_l=0$, 
we know that 
\begin{equation*}
\Phi_l=\frac{\sum_{j\in N(l)}g_{lj}(\rho) \Phi_j}{\sum_{j\in N(l)}g_{lj}(\rho)}=
c+\frac{\sum_{j\in N(l)}g_{lj}(\rho) (\Phi_j-c)}{\sum_{j\in N(l)}g_{lj}(\rho)}<c,
\end{equation*}
which contradicts $\Phi_l=c$.
\end{proof}
This identification map induces a scalar inner product on $\mathcal{P}_o(G)$.
\begin{definition}\label{d9}
For any two tangent vectors $\sigma^1,\sigma^2\in T_\rho\mathcal{P}_o(G)$, 
we define an inner product 
$g :T_\rho\mathcal{P}_o(G)\times T_\rho\mathcal{P}_o(G)  \rightarrow \mathbb{R}$:
\begin{equation}\begin{split}\label{formula}
g(\sigma^1,\sigma^2):=\sum_{i=1}^n\sigma^1_i\Phi^2_i=\sum_{i=1}^n\Phi^1_i\sigma^2_i=\frac{1}{2\Delta x^2}\sum_{ (i,j)\in E}g_{ij}(\rho)(\Phi_i^1-\Phi_j^1)(\Phi_i^2-\Phi_j^2),
\end{split} 
\end{equation}
where $[\Phi^1]$, $[\Phi^2]\in \mathbb{R}^n/\sim$, are such that 
$\sigma^1=\tau([\Phi^1])$, $\sigma^2=\tau([\Phi^2])$. 
\end{definition}

Under the above setting, we have 
\begin{equation*}
\left(W_2(\rho^0, \rho^1)\right)^2=\inf\{\int_0^1 g(\sigma, \sigma)dt~:~ 
\frac{d\rho}{dt}=\sigma,\quad \rho(0)=\rho^0,\quad \rho(1)=\rho^1, \quad \rho\in\mathcal{C}\},
\end{equation*} 
where $\mathcal{C}$ is the set of all continuously differentiable curves $\rho: [0,1]\rightarrow \mathcal{P}_o(G)$.
So, the metric is well defined, see more details in \cite{chow2012}. 

\subsection{Derivation of the scheme}
Based on the metric manifold $(\mathcal{P}_o(G), W_{2})$, we now derive the 
semi-discretization scheme \eqref{a1} as gradient flow of the discrete free energy. 

In abstract form, the gradient flow is defined by 
\begin{equation*}\label{gfd}
\frac{d\rho}{dt}=-\textrm{grad}_{\mathcal{P}_o(G)} \mathcal{F}(\rho).
\end{equation*}
Here $\textrm{grad}\mathcal{F}(\rho)$ is in the tangent space $T_\rho\mathcal{P}_o(G)$, 
which is defined by the duality condition: 
\begin{equation*}
g(\textrm{grad}_{\mathcal{P}_o(G)}\mathcal{F}(\rho), \sigma)=
\textrm{d}{\mathcal{F}}(\rho)\cdot{ \sigma,}\quad 
\textrm{for any}~ \sigma\in T_\rho\mathcal{P}_o(G),
\end{equation*}
where $\textrm{d}\mathcal{F}\cdot\sigma=
\sum_{i=1}^n\frac{\partial}{\partial \rho_i}\mathcal{F}(\rho)\sigma_i$.
Hence the gradient flow satisfies
\begin{equation}\label{gfdd}
(\frac{d\rho}{dt}, \sigma)_{\rho}+\textrm{d}\mathcal{F}(\rho)\cdot{\sigma}=0,
\quad \textrm{for any}~ \sigma\in T_\rho\mathcal{P}_o(G).
\end{equation}
Following \eqref{gfdd}, we derive \eqref{a1} in Theorem \ref{Derive} below.  
\begin{theorem}\label{Derive}
Given a graph $G$, a constant $\beta> 0$, a vector $(v_i)_{i=1}^n$ and a
symmetric matrix $(w_{ij})_{1\leq i, j\leq n}$.
Then the gradient flow of the discrete free energy 
\begin{equation*}
\mathcal{F}(\rho)=\sum_{i=1}^nv_i\rho_i+\frac{1}{2}\sum_{i=1}^n\sum_{j=1}^n
w_{ij}\rho_i\rho_j+\beta\sum_{i=1}^n\rho_i\log\rho_i,
\end{equation*}
on the metric manifold $(\mathcal{P}_o(G), W_{2})$, is  
\begin{equation*}
\begin{split}
\frac{ d\rho_i}{dt}=\frac{1}{\Delta x^2}\{\sum_{j\in N(i)} \rho_j( F_j(\rho)- F_i(\rho))_+
-\sum_{j\in N(i)}\rho_{i}( F_i(\rho)-F_j(\rho))_+\},
\end{split}
\end{equation*}
for any $i\in V$. Here $F_i(\rho)=\frac{\partial}{\partial\rho_i}\mathcal{F}(\rho)=
v_i+\sum_{j=1}^nw_{ij}\rho_j+\beta\log\rho_i+\beta.$
\end{theorem}
\begin{proof}[Proof of Theorem \ref{Derive}]
We show the derivation of \eqref{a1}. For any $\sigma\in T_\rho\mathcal{P}_o(G)$, 
there exists $[\Phi]\in \mathbb{R}^n/\sim$, such that 
$\tau([\Phi])=\sigma$. On one hand, we denote $\frac{d\rho}{dt}=(\frac{d\rho_i}{dt})_{i=1}^n$. 
From definition \ref{d9}, 
\begin{equation}\label{d1}
\begin{split}
(\frac{d\rho}{dt}, \sigma)_\rho=&\sum_{i=1}^n\frac{d\rho_i}{dt}\Phi_i \ .
\end{split}
\end{equation}
At the same time, we also have
\begin{equation}\label{new1}
\begin{split}
\textrm{d}\mathcal{F}(\rho)\cdot{\sigma}
=&\sum_{i=1}^n\frac{\partial}{\partial \rho_i}\mathcal{F}(\rho)\cdot 
\sigma_i=\sum_{i=1}^nF_i(\rho)\frac{1}{\Delta x^2} \sum_{j\in N(i)}g_{ij}(\rho)(\Phi_i-\Phi_j)\\
=&\frac{1}{\Delta x^2} \{\sum_{i=1}^n\sum_{j\in N(i)}g_{ij}(\rho)F_i(\rho)\Phi_i- 
\sum_{i=1}^n\sum_{j\in N(i)}g_{ij}(\rho)F_i(\rho)\Phi_j\}\\
&\textrm{Relabel $i$ and $j$ on second formula}\\
=&\frac{1}{\Delta x^2}\{\sum_{(i,j)\in E}g_{ij}(\rho)F_i(\rho)\Phi_i- 
\sum_{(i,j )\in E}g_{ji}(\rho)F_j(\rho)\Phi_i\}\\
=&\frac{1}{\Delta x^2}\{\sum_{i=1}^{n}\sum_{j\in N(i)} g_{ij}(\rho)\big(F_i(\rho)-F_j(\rho)\big)\Phi_i\}. \\
\end{split}
\end{equation}
Combining \eqref{d1} and \eqref{new1} into \eqref{gfdd}, we have  
\begin{equation*}
\begin{split}
0=&(\frac{d\rho}{dt}, \sigma)_{\rho}+\textrm{d}\mathcal{F}(\rho)\cdot{\sigma} \\
=&\sum_{i=1}^n \{\frac{d\rho_i}{dt}+\frac{1}{\Delta x^2}\sum_{j\in N(i)} 
g_{ij}(\rho)\big(F_i(\rho)-F_j(\rho)\big)\}\Phi_i.
\end{split}
\end{equation*}
Since the above formula is true for all $(\Phi_i)_{i=1}^n\in \mathbb{R}^n$,  then
\begin{equation*}
\frac{d\rho_i}{dt}+\frac{1}{\Delta x^2}\sum_{j\in N(i)} g_{ij}(\rho)\big(F_i(\rho)-F_j(\rho)\big)=0
\end{equation*}
holds for all $i\in V$.
From the definition of $g_{ij}(\rho)$ in \eqref{gij}, we have \eqref{a1}.
\end{proof}

To summarize, we have introduced a new discretization, which can be formally represented as  
\begin{equation*}
\frac{d\rho}{dt}=\textrm{div}_G(\rho\nabla_G F(\rho)),\quad F(\rho)=
(\frac{\partial}{\partial \rho_i}\mathcal{F}(\rho))_{i=1}^n,
\end{equation*} 
where
\begin{equation*}
\textrm{div}_G (\rho\nabla_G F(\rho))=\frac{1}{\Delta x^2}\big(\sum_{j\in N(i)} 
\rho_j( F_j(\rho)- F_i(\rho))_+-\sum_{j\in N(i)}\rho_{i}( F_i(\rho)-F_j(\rho))_+\big)_{i=1}^n.
\end{equation*}

\subsection{Gradient flow properties}
Here, we show that \eqref{a1} is a well defined ODE with gradient flow structures. 
\begin{theorem}\label{existence}
For any initial condition $\rho^0\in\mathcal{P}_o(G)$, \eqref{a1} has a unique solution 
$\rho(t): [0,\infty)\rightarrow \mathcal{P}_o(G)$. Moreover,
\begin{itemize}
\item[(i)]  there exists a constant $c=c(\rho^0)>0$ depending on $\rho^0$, such that 
$\rho_i(t)\geq c$ for all $i\in V$ and $t>0$;
\item[(ii)] the free energy $\mathcal{ F}(\rho)$ is a Lyapunov function of \eqref{a1}:
\begin{equation*}
\frac{d}{dt}\mathcal{F}(\rho(t))=-\sum_{ (i,j)\in E}(\frac{ F_i(\rho)- F_j(\rho)}{\Delta x})_+^2\rho_i.
\end{equation*}
 Further, if $\lim_{t\rightarrow \infty}\rho(t)$ exists, call it 
 $\rho^{\infty}$, then $\rho^{\infty}$ is a Gibbs measure. 
\end{itemize}
\end{theorem}
\begin{proof}
The proof of (i) can be found in \cite{li_thesis}, which is just 
a slight modification of \cite{chow2012}. 
Below, we only show (ii), which justifies saying that \eqref{a1} is a gradient system.
Firstly, we show that $\mathcal{F}(\rho)$ is a Lyapunov function: \begin{equation*}
\begin{split}
\frac{d}{dt}\mathcal{F}(\rho(t))=&\sum_{i=1}^nF_i(\rho)\cdot \frac{d\rho_i}{dt}\\
=&\frac{1}{\Delta x^2}\{\sum_{i=1}^n\sum_{j\in N(i)}F_i(\rho)(F_j(\rho)-F_i(\rho))_+
\rho_j-\sum_{i=1}^n\sum_{j\in N(i)}F_i(\rho)(F_i(\rho)-F_j(\rho))_+\rho_i\} \\
&\textrm{Switch $i$, $j$ on the first formula}\\
=&\frac{1}{\Delta x^2}\{\sum_{i=1}^n\sum_{j\in N(i)}F_j(\rho)(F_i(\rho)-F_j(\rho))_+
\rho_i-\sum_{i=1}^n\sum_{j\in N(i)}F_i(\rho)(F_i(\rho)-F_j(\rho))_+\rho_i \}\\
=&-\sum_{(i,j)\in E}(\frac{F_i(\rho)-F_j(\rho)}{\Delta x})_+^2\rho_i\leq 0.
\end{split}
\end{equation*}

Secondly, we prove that if $\rho^{\infty}=\lim_{t\rightarrow \infty}\rho(t)$ exists, 
then $\rho^\infty$ is a Gibbs measure. Since $\lim_{t\rightarrow \infty}\frac{d\rho(t)}{dt}=0$, then $\lim_{t\rightarrow \infty}\frac{d}{dt}\mathcal{F}(\rho(t))=0$. 
From (i), we know that $\rho^{\infty}_i\geq c(\rho^0)>0$ for any $i\in V$;
so, the relation
\begin{equation*}
\sum_{i=1}^n\sum_{j\in N(i)}(F_i(\rho^{\infty})-F_j(\rho^{\infty}))_+^2\rho_i^{\infty}=0
\end{equation*} 
implies $F_i(\rho^{\infty})=F_j(\rho^{\infty})$ for any $(i,j)\in E$. 
Since the graph is strongly connected, 
\begin{equation*}
F_i(\rho^{\infty})=F_j(\rho^{\infty}),\quad \textrm{for any $i, j\in V$}.
\end{equation*}
Let 
\begin{equation*}
C:= v_i+\sum_{j=1}^nw_{ij}\rho_j^{\infty}+\beta\log\rho_i^{\infty} ,\quad 
{\textrm{which is constant for any $i\in V$,}}
\end{equation*}
$K=e^{-\frac{C}{\beta}}$ and use the fact $\sum_{i=1}^n\rho_i^{\infty}=1$. Then, we have
\begin{equation*}
\rho_i^{\infty}=\frac{1}{K}e^{-\frac{v_i+\sum_{j=1}^nw_{ij}\rho_j^{\infty}}{\beta}}, 
\quad K=\sum_{j=1}^ne^{-\frac{v_i+\sum_{j=1}^nw_{ij}\rho_j^{\infty}}{\beta}}.
\end{equation*}
Hence $\rho^{\infty}$ is a Gibbs measure, which finishes the proof.
\end{proof}
Notice that $(\mathcal{P}_o(G), W_2)$ is not a smooth Riemannian manifold, 
since for fixed $i$ and $j\in V$, $g_{ij}(\rho)$ may be discontinuous with respect to $\rho$.
Still, even though $(\mathcal{P}_o(G), W_2)$ is not smooth, \eqref{a1} is a well defined ODE 
for any initial condition $\rho^0\in \mathcal{P}_o(G)$.

One may be surprised by the unusual discretization of the Laplacian term, namely
\begin{equation}\label{LG}
\frac{1}{\Delta x^2}(\log\rho_j-\log\rho_i)g_{ij}(\rho)
\end{equation}
which is different from the commonly adopted centered difference. We call \eqref{LG} the ``Log-Laplacian.''
We observe that the Log-Laplacian plays a crucial role in the spatial discretization. 
Not only it implies that \eqref{a1}'s equilibria are Gibbs measures,   
but it also indicates that
the boundary of the probability set,
$$\partial \mathcal{P}(G)=\{(\rho_i)_{i=1}^n\mid \sum_{i=1}^n\rho_i=1,\quad 
\textrm{there exists some index $i$, such that $\rho_i=0$}\}\ ,$$ 
is a repeller for \eqref{a1}. 
We will see that this boundary repeller property plays an important role in
the convergence result of section \ref{convergence}. 

\section{Dissipation rate to a discrete Gibbs measure}\label{convergence}
Considering the gradient flow \eqref{a1}, an important question arises. 
Assuming that $\rho(t)$ converges to an equilibrium $\rho^{\infty}$, how fast is the convergence speed?
In the sequel, we show that the rate of convergence is exponential. 
Indeed, we capture such rate by the following explicit formula. 
\begin{definition}\label{def}
Denote 
\begin{equation*}
f_{ij}=\frac{\partial^2}{\partial \rho_i\partial \rho_j}\mathcal{F}(\rho),
\end{equation*}
and \begin{equation*}
h_{ij, kl}=\frac{f_{ik}+f_{jl}-f_{il}-f_{jk}}{\Delta x^2} \quad \textrm{for any $i$, $j$, $k$, $l\in V$}.
\end{equation*}
We define  
\begin{equation*}
\lambda_{\mathcal{F}}(\rho)=\min_{(\Phi_i)_{i=1}^n\in D} {\sum_{(i,j)\in E}
\sum_{(k,l)\in E}h_{ij, kl}(\frac{\Phi_i-\Phi_j}{\Delta x})_+\rho_i(\frac{\Phi_k-\Phi_l}{\Delta x})_+\rho_k},
\end{equation*}
where 
\begin{equation*}
D=\{(\Phi_i)_{i=1}^n\in\mathbb{R}^n\mid \sum_{(i,j)\in E}(\frac{\Phi_i-\Phi_j}{\Delta x})^2_+\rho_i=1\}.
\end{equation*}
\end{definition}
\begin{remark}
$\lambda_{\mathcal{F}}$ in Definition \ref{def} plays the role of the
smallest eigenvalue of the Hessian operator on Riemannian manifold of the
free energy at Gibbs measure; see \cite{li2017, li_thesis} for more details about this connection. 
\end{remark}
 
Based on $\lambda_\mathcal{F}(\rho)$, we show the exponential convergence result for \eqref{a1}. 
We will assume that $\rho^0$ is in the basin of attraction of $\rho^{\infty}$ for the gradient flow. 
I.e., if $\rho(t)$ is a solution of \eqref{a1} with initial condition $\rho^0$, then
\begin{equation*}
(A)\quad \lim_{t\rightarrow \infty}\rho(t)=\rho^{\infty}\quad \textrm{and} \quad 
\textrm{$\rho^{\infty}$ is an isolated equilibrium}.
\end{equation*}
\begin{theorem}\label{th12}
Let (A) hold, and let
$\lambda_{\mathcal{F}}(\rho^{\infty})>0$. Then there exists a constant 
$C=C(\rho^0,G)>0$, depending on $\rho^0$ and $G$, such that 
\begin{equation*}
\mathcal{F}(\rho(t))-\mathcal{F}(\rho^{\infty})\leq e^{-Ct}(\mathcal{F}(\rho^0)-\mathcal{F}(\rho^{\infty})).
\end{equation*}
Moreover, the asymptotic convergence rate is 
$2\lambda_{\mathcal{F}}(\rho^{\infty})$. I.e., for any sufficiently small 
$\epsilon>0$, there exists a time $T>0$ depending on $\epsilon$ and $\rho^0$, such that when $t>T$,
\begin{equation*}
\mathcal{F}(\rho(t))-\mathcal{F}(\rho^{\infty})\leq e^{-2(\lambda_{\mathcal{F}}(\rho^{\infty})-\epsilon)t}
(\mathcal{F}(\rho(T))-\mathcal{F}(\rho^{\infty})).
\end{equation*}
\end{theorem}
\begin{proof}[Motivation of the proof]
Our proof is motivated by some known facts of gradient flows in 
$\mathbb{R}^n$. We consider a $\lambda$-convex energy $g(x)\in C^2(\mathbb{R}^n)$, i.e. $\textrm{Hess}_{\mathbb{R}^n}g(x) \succeq \lambda I$, $\lambda>0$ for all $x\in \mathbb{R}^n$. The gradient flow associated to $g$ is 
\begin{equation*}
\frac{dx_t}{dt}=-\nabla g(x_t),\quad x_t\in\mathbb{R}^n.
\end{equation*}
We compare the first and second derivative of $g(x_t)$ with respect to $t$:
\begin{equation*}
\begin{split}
\frac{d}{dt}g(x_t)=&-(\nabla g(x_t), \nabla g(x_t)),\\
\frac{d^2}{dt^2}g(x_t)=&-2(\textrm{Hess}_{\mathbb{R}^n}g(x_t)\cdot \nabla g(x_t), 
\nabla g(x_t))\geq -2\lambda \frac{d}{dt}g(x_t).
\end{split}
\end{equation*}
From the above comparison, we obtain the convergence result.
Integrating on the time interval $[t, +\infty)$,
\begin{equation*}
\frac{d}{dt}[g(x_t)-g(x_\infty)]\leq -2\lambda [g(x_t)-g(x_\infty)],
\end{equation*}
and applying Gronwall's inequality, the energy function $g(x_t)$ decreases exponentially 
\begin{equation*}
g(x_t)-g(x_{\infty})\leq e^{-2\lambda t}(g(x_0)-g(x_\infty)).
\end{equation*}
In addition, from the dynamical viewpoint, the strict convexity of the
free energy can be weakened:
if the equilibrium $x_{\infty}$ is a strict local minimizer, the exponential convergence result 
is still valid. Furthermore, 
the asymptotic convergence rate is $2\lambda_{\min}\textrm{Hess}_{\mathbb{R}^n}g(x_\infty)$.
\end{proof}
\begin{proof}[Proof of Theorem \ref{th12}]
Motivated by the standard approach in $\mathbb{R}^n$, we briefly sketch our proof in Riemannian manifold $(\mathcal{P}_o(G), W_2)$; see \cite{li2017, li_thesis} for complete details. The main idea is to compare the first and second derivatives of the free energy along \eqref{a1}. 

\textbf{Claim}: 
\begin{equation}\label{argument}
\begin{split}
\frac{d^2}{dt^2}\mathcal{F}(\rho(t))=&\frac{2}{\Delta x^4}\sum_{(i,j)\in E}
\sum_{(k,l)\in E}h_{ij, kl}(F_i-F_j)_+\rho_i(F_k-F_l)_+\rho_k\\
&+o(\frac{d}{dt}\mathcal{F}(\rho(t))).
\end{split}
\end{equation}
Here we denote $\lim_{h\rightarrow 0}\frac{o(h)}{h}=0$, 
$F_i=\frac{\partial}{\partial \rho_i}\mathcal{F}(\rho)$, 
$f_{ij}=\frac{\partial^2 }{\partial \rho_i\partial \rho_j} \mathcal{F}(\rho)$ and 
$h_{ij, kl}=f_{ik}+f_{lj}-f_{il}-f_{jk}$.
If \eqref{argument} holds, it is not hard to check that Theorem \ref{th12} holds.

Let's show \eqref{argument} directly.
Recall the gradient flow \eqref{a1}
\begin{equation*}
\frac{d\rho_i}{dt}=\frac{1}{\Delta x^2}\{\sum_{j\in N(i)}(F_j-F_i)_+\rho_j-\sum_{j\in N(i)}(F_i-F_j)_+\rho_i\}.
\end{equation*} 
We compute the first derivative of the free energy along \eqref{a1}:
\begin{equation*}
\begin{split}
\frac{d}{dt}\mathcal{F}(\rho(t))=-\sum_{i=1}^n\sum_{j\in N(i)}(\frac{F_i-F_j}{\Delta x})_+^2\rho_i.
\end{split}
\end{equation*}
Notice that $\frac{d^2}{dt^2}\mathcal{F}(\rho(t))$
exists for all $t\geq 0$, because $(F_i(\rho)-F_j(\rho))_+^2$ is differentiable everywhere with respect to $\rho$.
Then we obtain the second derivative by using the product rule:
\begin{equation*}
\begin{split}
\frac{d^2}{dt^2}\mathcal{F}(\rho(t))=& \quad-\sum_{i=1}^n\sum_{j\in N(i)}(\frac{F_i-F_j}{\Delta x})_+^2
\frac{d\rho_i}{dt} \hspace{2.2cm} \clubsuit \\
&-2\frac{1}{\Delta x^2}\sum_{i=1}^n\sum_{j\in N(i)}(\frac{d F_i}{dt}-\frac{dF_j}{dt})(F_i-F_j)_+\rho_i.
\quad \spadesuit\\
\end{split}
\end{equation*}

Hence, \eqref{argument} can be shown by the following two steps. 
Firstly,  since $\rho(t)$ is assumed to converge to an equilibrium $\rho^{\infty}$ and 
the boundary is a repeller (Theorem \ref{existence}),
we know that $\frac{d\rho}{dt}\rightarrow 0$ while $\rho_i(t)\geq c(\rho^0)>0$. 
Hence $\clubsuit$ is a high order term of the first derivative
\begin{equation*}
\clubsuit=o(\frac{d}{dt}\mathcal{F}(\rho(t))).
\end{equation*}
Secondly, we have the following Lemma. 
 \begin{lemma}\label{spade}
\begin{equation*}
\spadesuit=2\sum_{(i,j)\in E}\sum_{(k,l)\in E}h_{ij,kl}(\frac{F_i-F_j}{\Delta x})_+
\rho_i(\frac{F_k-F_l}{\Delta x})_+\rho_k. 
\end{equation*}
\end{lemma}
\begin{proof}[Proof of Lemma \ref{spade}]
We derive this result by a direct calculation. Here we use the relabeling technique 
heavily: For a matrix $(k_{ij})_{1\leq i, j\leq n}$,
\begin{equation*}
\sum_{(i,j)\in E}k_{ij}=\sum_{(j,i)\in E}k_{ji}.
\end{equation*}
Then
\begin{equation*}
\begin{split}
-\frac{1}{2}\spadesuit=&\frac{1}{\Delta x^2}\sum_{i=1}^n\sum_{j\in N(i)}(F_i-F_j)_+
\rho_i(\frac{d}{dt}F_i(\rho(t))-\frac{d}{dt} F_j(\rho(t)))\\
=&\frac{1}{\Delta x^2} \sum_{i=1}^n\sum_{j\in N(i)}(F_i-F_j)_+
\rho_i(\sum_{k=1}^n\frac{\partial F_i}{\partial \rho_k}\frac{d\rho_k}{dt}-
\sum_{k=1}^n\frac{\partial F_j}{\partial \rho_k}\frac{d\rho_k}{dt})\\
=&\frac{1}{\Delta x^2}\sum_{i=1}^n\sum_{j\in N(i)}(F_i-F_j)_+
\rho_i\sum_{k=1}^n(f_{ik}-f_{kj})\frac{d\rho_k}{dt}\\
=&\frac{1}{\Delta x^4}\sum_{i=1}^n\sum_{j\in N(i)}(F_i-F_j)_+
\rho_i\sum_{k=1}^n(f_{ik}-f_{kj})[\sum_{l\in N(k)}(F_l-F_k)_+
\rho_l-\sum_{l\in N(k)}(F_k-F_l)_+\rho_k]\\
=&\frac{1}{\Delta x^4}\sum_{i=1}^n\sum_{j\in N(i)}(F_i-F_j)_+
\rho_i\{\sum_{k=1}^n\sum_{l\in N(k)}(f_{ik}-f_{kj})(F_l-F_k)_+\rho_l\\
&\hspace{4cm}-\sum_{k=1}^n\sum_{l\in N(k)}(f_{ik}-f_{kj})(F_k-F_l)_+\rho_k\}\\
&\textrm{ Relabel $k$, $l$ in the first formula}\\
=&\frac{1}{\Delta x^4}\sum_{i=1}^n\sum_{j\in N(i)}(F_i-F_j)_+
\rho_i\{\sum_{k=1}^n\sum_{l\in N(k)}(f_{il}-f_{lj})(F_k-F_l)_+\rho_k\\
&\hspace{4cm}-\sum_{k=1}^n\sum_{l\in N(k)}(f_{ik}-f_{kj})(F_k-F_l)_+\rho_k\}\\
=&\frac{1}{\Delta x^4}\sum_{i=1}^n\sum_{j\in N(i)}\sum_{k=1}^n\sum_{l\in N(k)}
(f_{il}-f_{lj}-f_{ik}+f_{kj})(F_i-F_j)_+\rho_i(F_k-F_l)_+\rho_k\\
=&\frac{1}{\Delta x^4}\sum_{(i,j)\in E}\sum_{(k,l)\in E}
(f_{il}-f_{lj}-f_{ik}+f_{kj})(F_i-F_j)_+\rho_i(F_k-F_l)_+\rho_k.
\end{split}
\end{equation*}
\end{proof}
Combining all the above facts, the claim and the proof of Theorem \ref{th12} follow.
\end{proof}

\subsection{Analysis of dissipation rate}
In the sequel, we further elucidate the relationship between convexity
of the free energy (Hessian operator in $\mathbb{R}^n$) and the dissipation rate.

\begin{lemma}\label{rate}
Denote
\begin{equation*}
\tilde{\textrm{div}}_G(\rho\nabla_G\Phi):=\big(\frac{1}{\Delta x^2}\{\sum_{j\in N(i)}(\Phi_i-\Phi_j)_+
\rho_i-\sum_{j\in N(i)}(\Phi_j-\Phi_i)_+\rho_j\}\big)_{i=1}^n.
\end{equation*}
Then $\lambda_{\mathcal{F}}(\rho)$ in Definition \ref{def} is equivalent to 
\begin{equation*}
\lambda_{\mathcal{F}}(\rho)=\min \{\big(\tilde{\textrm{div}}_G(\rho\nabla_G\Phi)\big)^T 
\textrm{Hess}_{\mathbb{R}^n}\mathcal{F}(\rho) \tilde{\textrm{div}}_G(\rho\nabla_G\Phi)
~:~ \sum_{(i,j)\in E}(\frac{\Phi_i-\Phi_j}{\Delta x})^2_+\rho_i=1\}.
\end{equation*}
\end{lemma}
The proof of Lemma \ref{rate} is based on a direct computation, see details in page 42 of \cite{li_thesis}. Lemma \ref{rate} gives convergence rates for many semi-discretization schemes.
\begin{corollary}\label{corollary4}
Consider the gradient flow \eqref{a1} of the free energy 
\begin{equation*}
\mathcal{F}(\rho)=\sum_{i=1}^n v_i\rho_i+\frac{1}{2}\sum_{i=1}^n\sum_{j=1}^n
w_{ij}\rho_i\rho_j+\beta\sum_{i=1}^n\rho_i\log\rho_i.
\end{equation*}
If the matrix $W=(w_{ij})_{1\leq i, j\leq n}$ is semi positive definite, then there
is a unique Gibbs measure
$\rho^{\infty}$, which is a global attractor of \eqref{a1}. 
Moreover, there exists a constant $C>0$, such that 
\begin{equation*}
\mathcal{F}(\rho(t))-\mathcal{F}(\rho^{\infty})\leq 
e^{-Ct}(\mathcal{F}(\rho^0)-\mathcal{F}(\rho^{\infty}))
\end{equation*}
with asymptotic rate $2\lambda_{\mathcal{F}}(\rho^{\infty})$.
\end{corollary}
\begin{proof}
The main idea of proof is as follows (full details are in \cite{li_thesis}).
Notice that since $\textrm{Hess}_{\mathbb{R}^n}\sum_{i=1}^n\rho_i\log\rho_i=
\textrm{diag}(\frac{1}{\rho_k^{\infty}})_{1\leq k\leq n}$
and the matrix $W$ is semi positive definite, then
\begin{equation*}
\textrm{Hess}_{\mathbb{R}^n}\mathcal{F}(\rho)|_{\rho=\rho^{\infty}}=
W+\beta\textrm{diag}(\frac{1}{\rho_k^{\infty}})_{1\leq k\leq n}
\end{equation*}
is a positive definite matrix. 
Then, from Lemma \ref{rate} and Theorem \ref{th12}, we know that \eqref{a1} converges exponentially. 
\end{proof}

Throughout this section, we observe another important effect of the
Log-Laplacian, which reflects the convexity property of the linear entropy
\begin{equation*}
\mathcal{H}(\rho)=\sum_{i=1}^n\rho_i\log\rho_i.
\end{equation*}
Lemma \ref{rate} says that 
\begin{equation*}
\lambda_{\mathcal{H}}(\rho)=\min\{\sum_{i=1}^n\frac{1}{\rho_i}
(\tilde{\textrm{div}}_G(\rho\nabla_G\Phi)|_i)^2~:~\sum_{(i,j)\in E}
(\frac{\Phi_i-\Phi_j}{\Delta x})^2_+\rho_i=1\}.
\end{equation*}
Given any Gibbs measure $\rho^{\infty}$, we know that $\lambda_{\mathcal{H}}(\rho^{\infty})>0$. 
To visualize that, consider a simple example with no interaction energy, meaning 
that $(w_{ij})=0$. In this case, \eqref{a1} is a semi-discretization for a linear 
Fokker-Planck equation. 
The free energy is 
$$\mathcal{F}(\rho)=\sum_{i=1}^nv_i\rho_i+\beta\mathcal{H}(\rho).$$ 
Here, strict convexity of $\mathcal{H}(\rho)$ tells  
that there always exists a constant $C>0$, such that 
\begin{equation*}
\mathcal{F}(\rho(t))-\mathcal{F}(\rho^{\infty})\leq e^{-Ct}(\mathcal{F}(\rho^0)-\mathcal{F}(\rho^{\infty}))
\end{equation*}
holds with asymptotic rate $2\lambda_{\mathcal{H}}(\rho^{\infty})$. 

\section{Numerical analysis}\label{na}
In this section,  we show some numerical properties of \eqref{a1}.

\subsection{Spatial consistency}
To begin with, we show that \eqref{a1} is a finite volume scheme for the PDE \eqref{PDE}. 
For concreteness, we use a lattice graph.  Rewrite \eqref{a1} in  the following form
\begin{equation*}
\frac{d \rho_i}{dt}=\frac{1}{\Delta x^2}\{\sum_{v=1}^d\sum_{j\in N_v(i)}[F_j(\rho)- F_i(\rho)]_+
\rho_j-\sum_{v=1}^d\sum_{j\in N_v(i)}[F_i(\rho)-F_j(\rho)]_+\rho_i\}.
\end{equation*}
Denote $i=(i_1, \cdots, i_d)$, and $G$ is a cartesian graph of $d$ one dimensional lattices,
i.e. $G=G_1\Box \cdots \Box G_d$ with $G_v=(V_v, E_v)$. Here
\begin{equation*}
N_v(i)=\{(i_1,\cdots, i_{v-1}, j_v, i_{v+1},\cdots, i_d)\in V\mid (i_v,j_v)\in E_v\}.
\end{equation*}
\begin{theorem}\label{semi}
The semi-discretization \eqref{a1} is a consistent finite volume scheme for the PDE \eqref{PDE}.
\end{theorem}
\begin{proof} 
Denote by $\rho_i(t)$ a discrete probability function 
\begin{equation*}
\rho_i(t)=\int_{C_i}\rho(t,x)dx,
\end{equation*}
where $C_i$ is a cube in $\mathbb{R}^d$ centered at point $i$ with equal 
width $\Delta x$. Here $i\in V$ represents a point $x(i)\in \mathbb{R}^d$. 
Let $e_v=(0,\cdots,1,\cdots, 0)^T$, where $1$ is in the $v$-th position, 
$v=1, \cdots, d$. So in this setting, $N_v(i)$ for a lattice graph only contains the
two points $x(i)-e_v\Delta x $, $x(i)+e_v\Delta x $. Denote $\rho_j(t)$ by 
\begin{equation*}
\rho_j(t)=\int_{C_{i_+}}\rho(t,x)dx,
\end{equation*}
where $j\in N(i)$ satisfies 
$x(j)=x(i)+e_v\Delta x$ and $C_{i_+}$ is a cube centered at the point $j\in V$.

Without loss of generality, we assume 
$F(x(i)+e_v\Delta x,\rho)\geq F(x(i),\rho)\geq F(x(i)-e_v\Delta x, \rho)$. 
Applying Taylor expansion of \eqref{a1} relative to the direction $e_v$, we obtain
\begin{equation}\label{i}
\begin{split}
&\frac{1}{\Delta x^2}\{\sum_{j\in N_v(i)}[F_j(\rho)- F_i(\rho)]_+\rho_j-
\sum_{j\in N_v(i)}[F_i(\rho)-F_j(\rho)]_+\rho_i\}\\
=&\frac{1}{\Delta x^2}\{[F(x(i)+e_v\Delta x,\rho)- F(x(i),\rho)]\int_{C_{i_+}}\rho(t,x)dx\\
&\hspace{0.5cm}- [F(x(i),\rho)-F(x(i)-e_v\Delta x,\rho)]  \int_{C_{i}}\rho(t,x)dx \} \\
=&\quad\frac{1}{\Delta x^2}\{[\frac{\partial F}{\partial x_v}(x(i), \rho)
\Delta x+\frac{1}{2}\frac{\partial F}{\partial x_v}(x(i),\rho)\Delta x^2]\int_{C_{i_+}}\rho(t,x)dx\\
&\hspace{0.9cm}-[\frac{\partial F}{\partial x_v}(x(i),\rho)\Delta x-
\frac{1}{2}\frac{\partial }{\partial x_v}F(x(i),\rho )\Delta x^2]\int_{C_{i}}\rho(t,x)dx+O(\Delta x^3)\}\\
=&\quad \frac{1}{\Delta x}\frac{\partial F}{\partial x_v}(x(i), \rho)
[\int_{C_{i_+}}\rho(t,x)dx-\int_{C_{i}}\rho(t,x)dx]\\
&+\frac{1}{2}\frac{\partial^2 F}{\partial x_v^2}(x(i), \rho)[\int_{C_{i_+}}\rho(t,x)dx+
\int_{C_{i}}\rho(t,x)dx]+O(\Delta x)\\
=&\quad \frac{\partial F}{\partial x_v}(x(i),\rho)
\int_{C_{i}}\frac{\rho(t,x+e_v\Delta x)-\rho(t,x)}{\Delta x}dx\\
&+\frac{\partial^2 F}{\partial x_v^2}(x(i),\rho)
\int_{C_{i}}\frac{\rho(t,x+e_v\Delta x)+\rho(t,x)}{2}dx+O(\Delta x)\\
=&\int_{C_{i}}\nabla_{x_v}\cdot \big(\rho(t,x)\nabla_{x_v} F(x,\rho)\big)dx+O(\Delta x)\ .
\end{split}
\end{equation}
Similarly, we can show the same results for other possible configurations, such as 
$F(x(i)-e_v\Delta x,\rho)\geq F(x(i),\rho)\geq F(x(i)+e_v\Delta x, \rho)$, 
$F(x(i),\rho)\geq F(x_v-e_v\Delta x,\rho)\geq F(x(i)+e_v\Delta x, \rho)$.

Therefore, combining all directions $e_v$ with $v=1,\cdots, d$, the right-hand-side of \eqref{a1} becomes 
\begin{equation*}
\begin{split}
&\frac{d\rho_i}{dt}-\frac{1}{\Delta x^2}\sum_{v=1}^d\{\sum_{j\in N_v(i)}
[F_j(\rho)- F_i(\rho)]_+\rho_j-\sum_{j\in N_v(i)}[F_i(\rho)-F_j(\rho)]_+\rho_j\}\\
=&\int_{C_{i}}\{\frac{\partial\rho(t,x) }{\partial t}-
\sum_{v=1}^d\nabla_{x_v}\cdot \big(\rho(t,x)\nabla_{x_v} F(x,\rho)\big)\}dx+dO(\Delta x)\\
=&\int_{C_{i}}\{\frac{\partial\rho(t,x) }{\partial t}-
\nabla\cdot \big(\rho(t,x)\nabla_{x} F(x,\rho)\big)\}dx+dO(\Delta x)\\
=&O(\Delta x).
\end{split}
\end{equation*}
This shows that \eqref{a1} is a finite volume first order semi-discretization scheme for \eqref{PDE}.
\end{proof}
\subsection{Time discretization}
To deal with the time discretization, we use a forward Euler scheme on \eqref{a1}:
\begin{equation}\label{Euler}
\frac{\rho_i^{k+1}-\rho_i^k}{\Delta t}=\frac{1}{\Delta x^2}
\{\sum_{j\in N(i)} \rho_j^k( F_j(\rho^k)- F_i(\rho^k))_+
-\sum_{j\in N(i)}\rho_{i}^k( F_i(\rho^k)-F_j(\rho^k))_+\}.
\end{equation} 
\begin{lemma}
Assume that the discrete free energy $\mathcal{F}(\rho)$ is strictly convex on $\mathcal{P}_o(G)$. 
\begin{itemize}
\item[(i)] For a given small tolerance constant $\epsilon>0$, and initial measure 
$\rho^0\in \mathcal{P}_o(G)$, there exists a finite time 
$T=O(\log\frac{1}{\epsilon})$, such that when $t>T$,
\begin{equation*}
|\mathcal{F}(\rho(t))-\mathcal{F}(\rho^{\infty})|<\epsilon.
\end{equation*}
\item[(ii)] There exists a constant $h$, such that if $0<\Delta t\leq h$, 
$\rho^k=(\rho^k_i)_{i=1}^n\in \mathcal{P}_o(G)$, for all $k=0,1,\cdots, [\frac{T}{\Delta t}]$,
where $T$ is the value from (i).
\end{itemize}
\end{lemma}
\begin{proof}
(i) can be shown by the exponential convergence result in Corollary \ref{corollary4}. 
Since there exists a constant $C>0$, such that 
\begin{equation*}
\mathcal{F}(\rho(T))-\mathcal{F}(\rho^{\infty})\leq e^{-CT}(\mathcal{F}(\rho^0)-\mathcal{F}(\rho^{\infty})),
\end{equation*}
then if $\rho(T)$ satisfies $|\mathcal{F}(\rho(T))-\mathcal{F}(\rho^{\infty})|<\epsilon$, we need to set
\begin{equation*}
T\geq \frac{1}{C}\log \frac{\mathcal{F}(\rho^0)-\mathcal{F}(\rho^{\infty})}
{\mathcal{F}(\rho(T))-\mathcal{F}(\rho^{\infty})}.
\end{equation*}
In other words, we can approximate $\rho^{\infty}$ with $O(\epsilon)$ precision by 
time $T=O(\log\frac{1}{\epsilon})$.

We prove (ii) in two steps. Firstly, we show that $\rho^k=(\rho_i^k)_{i=1}^n$ stays positive ($\min_{i\in V}\rho_i^k>0$)
for all $k=1,\cdots, N$. From Theorem \ref{Derive}, we know that the boundary is a repeller for \eqref{a1}. 
This means that there exists a constant $\epsilon_0=\epsilon_0(\rho^0) >0$, such that 
\begin{equation*}
\min_{i\in \{1,\cdots, n\}}\rho_i(t)\geq\epsilon_0(\rho^0), \quad \textrm{for all $t\geq 0$}.
\end{equation*}
Since the forward Euler scheme is convergent for Lipschitz right-hand-sides (and 
this is the case for us),
there exists constant $h$, such that when $\Delta t\le h$, we have 
\begin{equation*}
\min_{i\in\{1,\cdots, n\}}|\rho_i(k\Delta t)-\rho_i^k|\leq \frac{\epsilon_0}{2}, 
\end{equation*}
from which $\min_{i\in\{1,\cdots, n\}}\rho_i^k\geq \frac{1}{2}\epsilon_0>0$.

Secondly, we show that $\sum_{i=1}^n\rho^k_i=1$ for all $k=1,\cdots, N$. 
Since $\sum_{i=1}^n \rho^0=1$, it is sufficient to prove that
\begin{equation*}
\sum_{i=1}^n\rho_i^{k+1}=\sum_{i=1}^n\rho_i^k,\quad \textrm{for any $k$.}
\end{equation*}
This is a linear invariant, and it is therefore kept by Euler method.  Indeed,
an explicit computation gives
\begin{equation*}
\begin{split}
\sum_{i=1}^n\frac{\rho_i^{k+1}-\rho_i^k}{\Delta t}
=&\sum_{i=1}^n\frac{1}{\Delta x^2}\{\sum_{j\in N(i)} \rho_j^k( F_j(\rho^k)- F_i(\rho^k))_+
-\sum_{j\in N(i)}\rho_{i}^k( F_i(\rho^k)-F_j(\rho^k))_+\}\\
=&0\ .
\end{split}
\end{equation*}
\end{proof}

\begin{remark}
In practice, cfr. with \cite{c2015}, we may 
consider
$\Delta t \leq \frac{\Delta x^2}{\Delta(G)M}$, with 
$M=2\sup_{i\in V]}|F_i(\rho^k)|$ and $\Delta (G)$ representing the maximal degree of the graph $G$.
For sufficiently small $\Delta t$, we know that $M$ will be a bounded function
up to a finite time $T$. 
\end{remark}

\subsection{An extension}
We extend the idea of semi-discretization scheme \eqref{a1} to deal 
with more general Fokker-Planck equations.  
Consider
\begin{equation}\label{general}
\frac{\partial \rho}{\partial t}=\nabla \cdot [\rho \big(f_v(x,\rho)\big)_{v=1}^d].
\end{equation}
Here, \eqref{general} may fail to be a gradient flow with respect to the 2-Wasserstein metric. 
In this case, we cannot consider a discretization which is a gradient flow of a
certain free energy.  However, we can still  construct a flow (semi-discretization scheme)
whose solutions lie on the probability set. 
The observation to use is that there always exists functions $(u_v(x,\rho))_{i=1}^d$
such that 
\begin{equation*}
\nabla_{x_v}u_v(x,\rho)=f_v(x,\rho),\quad\textrm{for $v\in\{1,\cdots, d\}$}.
\end{equation*}
\begin{example}[van der Pol]
Consider the 2 dimensional Fokker-Planck equation
\begin{equation*}
\frac{\partial \rho}{\partial t}=-\nabla\cdot (\rho\begin{pmatrix} x_2\\ (1-x_1^2)-x_2  \end{pmatrix})+ 
\frac{\partial^2\rho }{\partial x_2^2}=-\nabla\cdot (\rho\begin{pmatrix} f_1(x,\rho)\\ f_2(x,\rho) \end{pmatrix}),
\end{equation*} 
where $x=(x_1, x_2)$, $f_1(x,\rho)=x_2$ and $f_2(x,\rho)=(1-x_1^2)-x_2 +\nabla_{x_2}\log\rho(x)$. We let 
\begin{equation*}
u_1(x,\rho)=\int f_1(x,\rho)dx_1=x_1x_2,
\end{equation*}
and 
\begin{equation*}
u_2(x,\rho)=\int f_2(x,\rho)dx_2=(1-x_1^2)x_2-\frac{1}{2}x_2^2+\log\rho(x_1,x_2).
\end{equation*}
Then the Fokker-Planck equation becomes 
\begin{equation*}
\frac{\partial \rho}{\partial t}=-\nabla\cdot (\rho\begin{pmatrix} \nabla_{x_1} u_1(x,\rho)\\ 
\nabla_{x_2}u_2(x,\rho) \end{pmatrix}).
\end{equation*}
\end{example}

Based on the above observation, we naturally extend \eqref{a1} to the semi-discretization
of \eqref{general} 
\begin{equation}\label{flow}
\frac{d \rho_i}{dt}=\frac{1}{\Delta x^2}\{\sum_{v=1}^d\sum_{j\in N_v(i)}[u_v(i,\rho)- u_v(j,\rho)]_+\rho_j
-\sum_{v=1}^d\sum_{j\in N_v(i)}[u_v(j,\rho)-u_v(i,\rho)]_+\rho_i\} .
\end{equation}
We observe that \eqref{a1} is a special case of \eqref{flow}.
Similarly to Theorem \ref{semi}, we can show that the semi-discretization \eqref{flow}
is a consistent finite volume scheme for \eqref{general}.

\section{Numerical experiments}\label{numerical}
In this section, we illustrate the proposed semi-discretization with several numerical experiments. 

\begin{example}[Nonlinear Fokker-Planck equation]\label{Exam2}
We consider a nonlinear interaction-diffusion equation in granular gas 
\cite{benedetto1998non, villani2002review},
\begin{equation*}\label{main}
\frac{\partial \rho}{\partial t}=\nabla\cdot[\rho\nabla\big(W*\rho+V(x)\big)]+\beta\Delta \rho,
\end{equation*}
where $W(x,y)=\frac{1}{3}\|x-y\|^3$ and $V(x)=\frac{\|x\|^2}{2}$ with $\|\cdot\|$ the 
2 norm in $\mathbb{R}^d$. 

The PDE has a unique stationary measure (Gibbs measure),  
\begin{equation*}
\rho^*(x)=\frac{1}{K}e^{-\frac{\int_{\mathbb{R}^d}W(x,y)\rho^*(y)dy+V(x)}{\beta}},\quad 
\textrm{where} \quad K=\int_{\mathbb{R}^d}e^{-\frac{\int_{\mathbb{R}^d}W(x,y)\rho^*(y)dy+V(x)}{\beta}}dx.
\end{equation*}

We apply \eqref{a1} to discretize this PDE with $\beta=0.01$:
\begin{equation*}
\begin{split}
\frac{ d\rho_i}{dt}=&\frac{1}{\Delta x^2}\{\sum_{j\in N(i)} \rho_j
( \sum_{i=1}^n w_{ij}\rho_i-\sum_{j=1}^nw_{ij}\rho_j+v_j-v_i +\beta\log\rho_j-\beta\log\rho_i)_+\\
&\quad -\sum_{j\in N(i)}\rho_{i}( \sum_{j=1}^nw_{ij}\rho_j-
\sum_{i=1}^nw_{ij}\rho_i+v_i-v_j+\beta \log\rho_i-\beta\log\rho_j)_+\},
\end{split}
\end{equation*}
and further discretize in time with
the forward Euler method \eqref{Euler} with time step $\Delta t =10^{-4}$
and initial condition $\rho^0_i=\frac{1}{L}e^{-\frac{\|x(i)\|^2}{200}}$, 
$L=\sum_{i=1}^ne^{-\frac{\|x(i)\|^2}{200}}$.

 When $d=2$, we consider a two dimensional lattice graph of $[-5, 5]\times[-5, 5]$ with 
$\Delta x=0.5$; see Figure \ref{Exam22}.
\begin{figure}
\centering
\subfloat[Gibbs measure]{\includegraphics[scale=0.4]{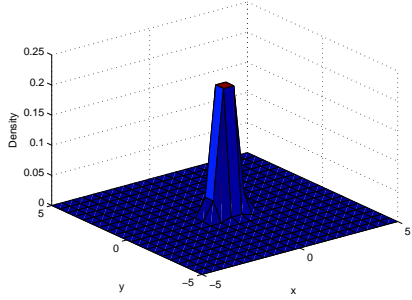}}\hspace{1cm}
\subfloat[Semi-log Y plot of $\mathcal{F}(\rho)-\mathcal{F}(\rho^\infty)$ w.r.t. iteration.]{\includegraphics[scale=0.3]{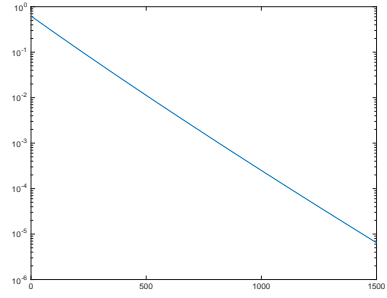}}
\caption{Example \ref{Exam2}: 2-d.}
\label{Exam22}
\end{figure}

It is known, see \cite{carrillo2003kinetic}, that solutions of this PDE converge
to the unique Gibbs measure, which itself converges to
a $\delta$-measure supported at the origin when $\beta\rightarrow 0$. 
In addition, the solution converges to the Gibbs measure exponentially.  We observe that
\eqref{a1} reflects all of
these behaviors and the free energy along solutions of \eqref{a1} decreases exponentially.  
\end{example}

\begin{example}[Linear Fokker-Planck equation]\label{LinearFP}
We consider a linear Fokker-Planck equation \begin{equation}\label{lvdp}
\frac{\partial \rho}{\partial t}=\nabla\cdot[\rho\nabla V(x)]+\beta\Delta \rho,
\end{equation}
with a potential function $V(x)=\frac{\|x\|^4}{4}-\frac{\|x\|^2}{2}$. 
Here the underlying state is $\mathbb{R}^d$. 
In this case, the unique Gibbs measure is.
\begin{equation*}
\rho^*(x)=\frac{1}{K}e^{-\frac{V(x)}{\beta}},\quad \textrm{where} \quad K=\int_{\mathbb{R}^d}e^{-\frac{V(x)}{\beta}}dx.
\end{equation*}

We use \eqref{a1} to approximate the solution of this PDE with $\beta=0.01$,
\begin{equation*}
\begin{split}
\frac{ d\rho_i}{dt}=&\frac{1}{\Delta x^2}\{\sum_{j\in N(i)} \rho_j( v_j-v_i +\beta\log\rho_j-\beta\log\rho_i)_+\\
&\quad -\sum_{j\in N(i)}\rho_{i}(v_i-v_j+\beta \log\rho_i-\beta\log\rho_j)_+\},
\end{split}
\end{equation*}
and further discretize in time by the forward Euler method \eqref{Euler} with time step 
$\Delta t =10^{-4}$.  Initial condition is
$\rho^0_i=\frac{1}{L}e^{-\frac{\|x(i)\|^2}{200}}$, $L=\sum_{i=1}^ne^{-\frac{\|x(i)\|^2}{200}}$.

If $d=2$, we take a uniform discretization of $[-5, 5]\times[-5, 5]$ with $\Delta x=0.5$;
see Figure \ref{Exam32}.
\begin{figure}
\centering
\subfloat[Gibbs measure]{\includegraphics[scale=0.4]{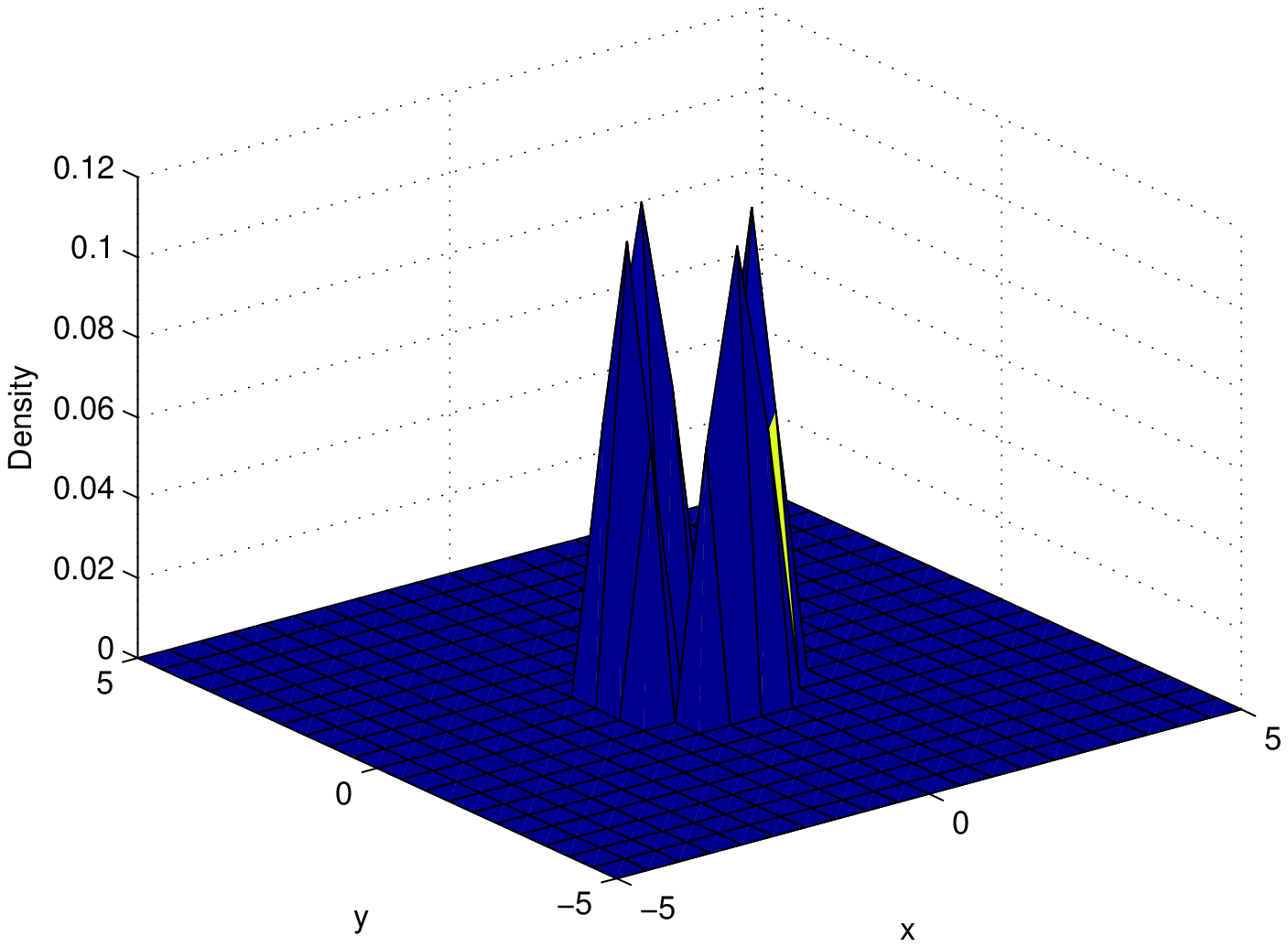}}\hspace{1cm}
\subfloat[Semi-log Y plot of $\mathcal{F}(\rho)-\mathcal{F}(\rho^\infty)$ w.r.t. iteration.]{\includegraphics[scale=0.3]{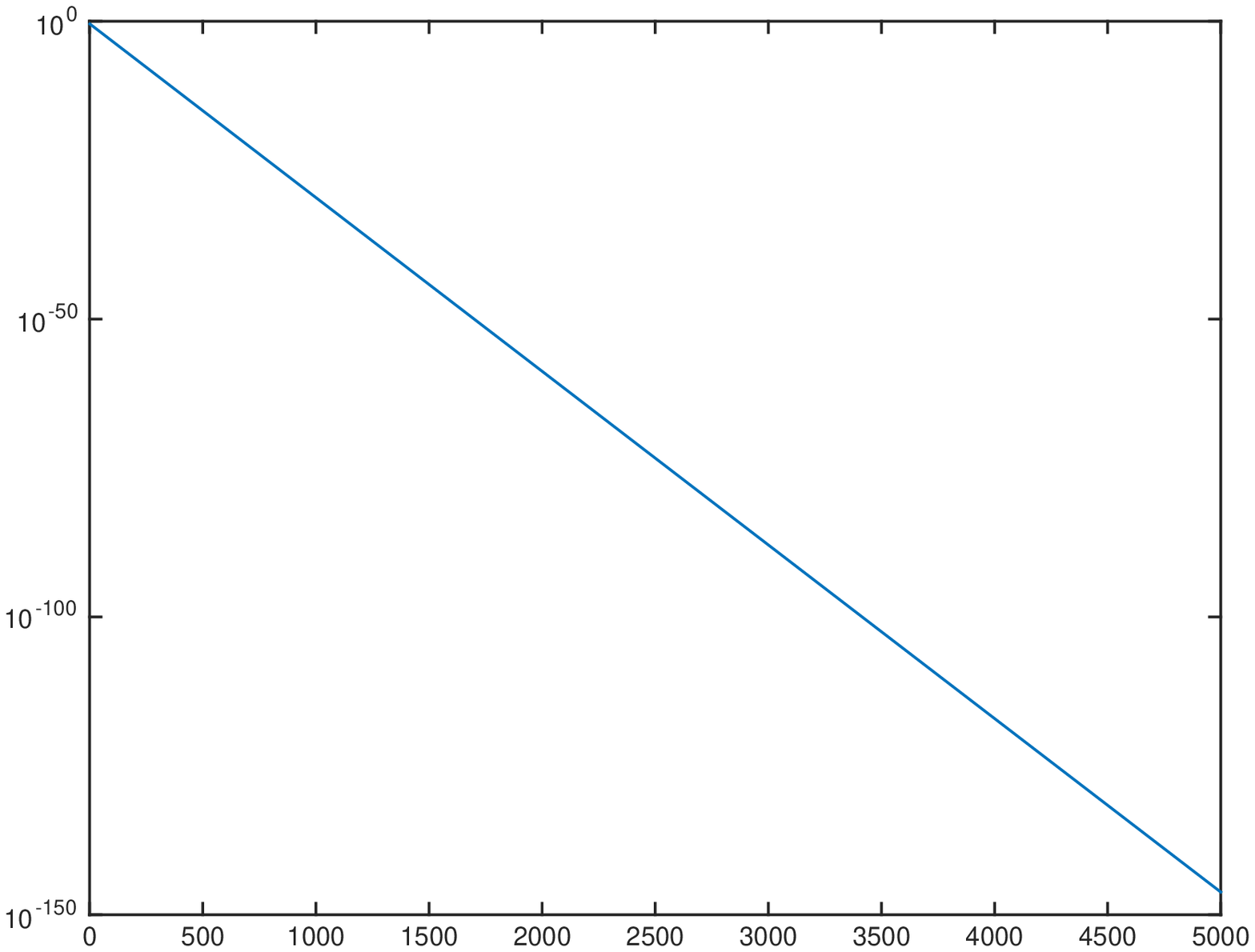}}
\caption{Example \ref{LinearFP}: 2-d.}
\label{Exam32}
\end{figure}

The computational results in both cases reflects that the linear Fokker-Planck equation always converges
to the Gibbs measure exponentially, which is in agreement with the discussion of Section \ref{convergence}. Note that
here the potential function $V(x)$ is not strictly convex.  It is the strict convexity of
the entropy in probability set that plays the key role in convergence. 
This asymptotic convergence rate is fully determined by 
$\lambda_{\mathcal{H}}(\rho^{\infty})$ in Definition \ref{def}. 
\end{example}


\begin{example}[General Fokker-Planck equation]\label{GeneralFP}
We consider the Fokker-Planck equation \cite{liVDP}
\begin{equation*}
\frac{\partial \rho}{\partial t}+\nabla\cdot (\rho\begin{pmatrix} x_2\\ 
(1-x_1^2)-x_2  \end{pmatrix})=\beta \Delta_{x_2}\rho,
\end{equation*}
whose underlying state is the stochastic van der Pol oscillator
\begin{equation*}
\begin{split}
dx_1&= x_2dt\\
dx_2&=[(1-x_1^2)x_2-x_1]dt+\sqrt{2\beta} dW_t.
\end{split}
\end{equation*}

We apply the semi-discretization \eqref{flow} to approximate the solution of this PDE.
Further, we discretize in time
by the forward Euler method \eqref{Euler} with time step $\Delta t =10^{-4}$.
Initial condition is 
$\rho^0_i=\frac{1}{L}e^{-\frac{\|x(i)\|^2}{200}}$, $L=\sum_{i=1}^ne^{-\frac{\|x(i)\|^2}{200}}$.

Let $\beta=0.125$, and consider a lattice graph on $[-10, 10]\times [-10, 10]$
with $\Delta x=0.4$.  The result in Figure \ref{Exam51} shows the obtained
approximation of the stationary measure of the stochastic van der Pol oscillator.  
\begin{figure}
\centering
{\includegraphics[scale=0.4]{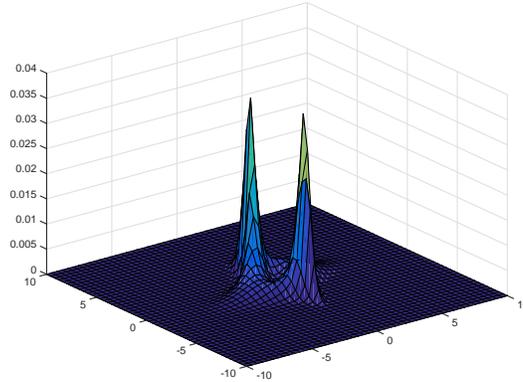}}
\caption{Example \ref{GeneralFP}.  Stationary measure, van der Pol.}
\label{Exam51}
\end{figure}

Similarly, we consider the Fokker-Planck equation
\begin{equation*}
\frac{\partial \rho}{\partial t}+\nabla\cdot (\rho\begin{pmatrix} x_2\\  
-2\xi\omega x_2+\omega x_1-\omega^2rx_1^3 \end{pmatrix})=\beta \Delta_{x_2}\rho,
\end{equation*}
associated with the stochastic Duffing oscillator
\begin{equation*}
\begin{split}
dx_1&= x_2dt\\
dx_2&=[-2\xi\omega x_2+\omega x_1-\omega^2rx_1^3]dt+\sqrt{2\beta} dW_t.
\end{split}
\end{equation*}

Let $\xi=0.2$, $\omega=1$, $r=0.1$, $\beta=0.125$ and a lattice graph of $[-10, 10]\times [-10, 10]$
with $\Delta x=0.4$. 
The computed invariant measure is shown in Figure \ref{Exam52}. 
\begin{figure}
\centering
{\includegraphics[scale=0.4]{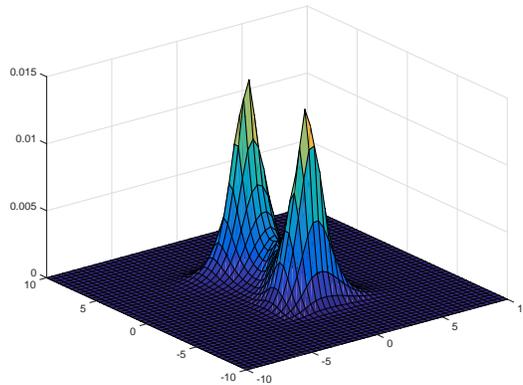}}
\caption{Example \ref{GeneralFP}.  Stationary measure, Duffing.}
\label{Exam52}
\end{figure}
\end{example}
In these examples, we have shown that our discretization scheme \eqref{flow} finds a
two-peaks stationary measure, even though the underlying Fokker-Planck equations are
not gradient flow type. It is interesting to observe that,
in the above two figures, stationary measures are supported around the limit cycles of 
the oscillators. 
The two peaks in the stationary measures reflect that there is slow and fast motion
in the underlying dynamical systems; namely, the two peaks are witness to the fact that
there is a larger probability that a trajectory at time $t$ will be found
in the slow motion region; see figure \ref{4}.
\begin{figure}
\centering
{\includegraphics[scale=0.4]{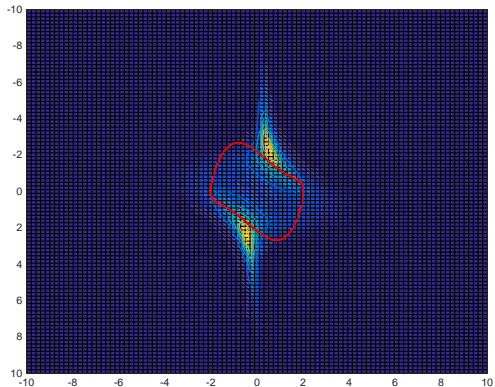}}
\caption{The plot of stationary measure and limit cycle (red) of van der Pol oscillator.}
\label{4}
\end{figure}
 
\section{Conclusion}
We have derived a new semi-discretization scheme \eqref{a1} for the PDE \eqref{PDE}. 
In comparison to other methods, our scheme \eqref{a1} has the following advantages. 

\begin{enumerate}
\item
Firstly, our scheme \eqref{a1} works on a finite graph, which is a spatial discretization 
of the underlying state. As a result of having this graph, we can handle a variety of
boundary conditions, e.g. zero-flux conditions or periodic conditions, and different 
types of underlying states, such as $\mathbb{R}^d$, open set of $\mathbb{R}^d$, or Riemannian manifold.  
\item
Secondly, we derive \eqref{a1} from the viewpoint of free energy and optimal transport. 
Hence, \eqref{a1} can keep the gradient flow structure of \eqref{PDE}.  
On one hand, this fact gives that \eqref{a1} is a well defined flow whose equilibria
are discrete Gibbs measures;
on the other hand, solutions of \eqref{a1} converge to a Gibbs measure with
exponential rate. This property allowed us to discretize \eqref{a1} in time
by a forward Euler scheme.
\item
Lastly, we bring a new twist to discretize the diffusion term, namely
\begin{equation*}
\frac{1}{\Delta x^2}\sum_{j\in N(i)}(\log\rho_j-\log\rho_i)g_{ij}(\rho).
\end{equation*}
We called it Log-Laplacian, and it is quite different from commonly known centered
differences or the Graph Laplacian. Although the log term brings some nonlinearities
into the algorithm, it also brings many benefits. One is that solutions of \eqref{a1}
always stay in $\mathcal{P}_o(G)$, and thus remain positive and conserve
the total probability automatically. The other is that the scheme naturally inherits 
the convexity of the entropy, a fact which plays a critical role in the convergence result.  
\end{enumerate}

\appendix
\setcounter{secnumdepth}{0}
\section{Appendix}\label{Append}
Generally, to obtain $\lambda_{\mathcal{F}}(\rho)$ in Definition \ref{def} is not easy. Below, we give simple 1-d model example to illustrate situations in which $\lambda_{\mathcal{F}}(\rho)$ can be explicitly obtained, and its dependence on the graph structure (the boundary conditions of the PDE).

\textbf{A 1-d model problem}.
Suppose that the free energy contains only the linear entropy term,
so that the gradient flow is the heat equation:
\begin{equation}\label{heat1-d}
\frac{\partial \rho}{\partial t}=\Delta \rho, \quad x\in (a,b).
\end{equation}
Here, we consider either 
(i) Neumann boundary conditions (zero flux) 
$\frac{\partial \rho}{\partial x}|_{x=a}=\frac{\partial \rho}{\partial x}|_{x=b}=0$, 
or (ii) periodic boundary conditions $\rho(t,a)=\rho(t,b)$.

We approximate the solution of \eqref{heat1-d} by \eqref{a1}, with a
uniform discretization $\Delta x=\frac{b-a}{n-1}$:
\begin{equation}\label{discrete-heat}
\frac{d\rho_i}{dt}=\frac{1}{\Delta x^2}\{\sum_{j\in N(i)}\rho_j(\log\rho_j-\log\rho_i)_+
-\sum_{j\in N(i)}\rho_i(\log\rho_i-\log\rho_j)_+\}.
\end{equation}
The above two types of boundary conditions lead to distinct graph structures. \\
(i) A lattice graph $L_{n}$: 
\begin{center}
  \begin{tikzpicture}[->,shorten >=1pt,auto,node distance=1cm,
        thick,main node/.style={circle,fill=blue!20,draw,minimum size=0.2cm,inner sep=1pt]} ]
   \node[main node] (1) {};
    \node[main node] (2) [right of=1]  {};
    \node[main node](3)[right of=2]{};
    \node[main node](4)[right of=3]{};
    \path[-]
    (1) edge node {} (2)
    (2) edge node{} (3)
    (3) edge node{} (4);        
\end{tikzpicture}
\end{center}
(ii) A cycle graph $C_n$: 
\begin{center}
\begin{tikzpicture}[ ->,shorten >=1pt,auto,node distance=3cm,
        thick,main node/.style={circle,fill=blue!20,draw,minimum size=0.2cm,inner sep=0pt]} ]
\def \n {6}
\def \radius {1}
\def \margin {0} 
\foreach \s in {1,...,\n}
{
  \node[draw, main node] (\s) at ({360/\n * (\s - 1)}:\radius) {};
}
    \path[-]
    (1) edge node {} (2)
    (2) edge node {} (3)
       (3) edge node {} (4)
    (4) edge node {} (5)
    (5) edge node {} (6)
    (6) edge node{} (1);
\end{tikzpicture}
\end{center}
In both cases, \eqref{discrete-heat} is the gradient flow of the discrete linear entropy
\begin{equation*}
\mathcal{H}(\rho)=\sum_{i=1}^n\rho_i\log\rho_i,
\end{equation*}
and the unique Gibbs measure is $\rho^{\infty}=(\frac{1}{n}, \cdots, \frac{1}{n})$. 
We are going to estimate how fast the solution $\rho(t)$ of the semi-discretization scheme 
\eqref{discrete-heat} converges to the equilibrium $\rho^{\infty}$. 

As we have seen in Theorem \ref{th12}, the asymptotic convergence rates are determined by 
$\lambda_{\mathcal{F}}(\rho)$:
\begin{equation}\label{CE}\begin{split}
\lambda_{\mathcal{H}}(\rho^{\infty})=\min_{\Phi\in \mathbb{R}^n}
\{ & \frac{1}{\Delta x^4}\sum_{(i,j)\in E}\sum_{(k,l)\in E}h_{ij, kl}(\Phi_i-\Phi_j)_+(\Phi_k-\Phi_l)_+~:~ \\
\qquad & \sum_{(i,j)\in E}(\frac{\Phi_i-\Phi_j}{\Delta x})^2_+\rho_i=1\},
\end{split}
\end{equation}
where 
\begin{equation*}
\textrm{$h_{ij, kl}=f_{ik}+f_{jl}-f_{il}-f_{jk}$, \quad and \quad
$f_{ij}(\rho^\infty)=\frac{\partial^2}{\partial\rho_i\partial\rho_j}
\mathcal{H}(\rho)|_{\rho=\rho^{\infty}}=\begin{cases} \frac{1}{\rho_i^{\infty}} \quad &\textrm{if $i=j$};\\
0\quad &\textrm{if $i\neq j$.}\quad  
\end{cases}$ }
\end{equation*}
For the present model, we can find exact values of \eqref{CE} for the above two graphs. 

\begin{theorem}\label{mainClaim}
We have
\begin{equation*}
\begin{split}
\lambda_{\mathcal{H}}(\rho^{\infty})=\frac{\pi^2}{(b-a)^2}+o(1), \quad (L_n)
\end{split}
\end{equation*}
and
\begin{equation*}
\begin{split}
\lambda_{\mathcal{H}}(\rho^{\infty})=\frac{4\pi^2}{(b-a)^2}+o(1).\quad (C_n)
\end{split}
\end{equation*}
\end{theorem}
\begin{proof}
First, consider the lattice graph $L_n$. 
Without loss of generality, let $(\Phi_i)_{i=1}^n$ in \eqref{CE} satisfy the relation
\begin{equation}\label{p1}
\Phi_1\geq \Phi_2 \geq \cdots \geq \Phi_n.
\end{equation}
Denote $\xi:=(\xi_i)_{i=1}^{n-1}\in\mathbb{R}^{n-1}_+$ by
\begin{equation}\label{p2}
\xi_i:=\frac{\Phi_{i+1}-\Phi_{i}}{\sqrt{n}\Delta x}, \quad 1\leq i\leq n,
\end{equation}
and substitute $\rho^{\infty}$ into \eqref{CE}, to obtain
\begin{equation*}
\lambda_{\mathcal{H}}(\rho^{\infty})=
\min_{\xi\in\mathbb{R}^{n-1}_+}\{\frac{1}{\Delta x^2}\xi^TA\xi\quad :\quad \xi^T\xi=1\},
\end{equation*}
where 
\begin{equation*}
A=\begin{pmatrix} 
  2 &  -1 &  &  \\ 
 -1 & 2 & -1 &  \\
\\
&\ddots & \ddots &\ddots &\\
\\
& & -1& 2 & -1\\
& & &-1 & 2
\end{pmatrix}\in \mathbb{R}^{(n-1)\times (n-1)}. 
\end{equation*}
It is simple to observe that $A$ is positive definite and 
that\footnote{Here the eigenvector of $A$ corresponding to the smallest
eigenvalue satisfies the assumption \eqref{p1}.} 
\begin{equation*}
\begin{split}
\lambda_{\mathcal{H}}(\rho^{\infty})=&\frac{1}{\Delta x^2}\times
(\textrm{the smallest eigenvalue of $A$})=\frac{1}{\frac{(b-a)^2}{(n-1)^2}}[2-2\cos(\frac{\pi}{n-1})]\\
=&\frac{\pi^2}{(b-a)^2}+o(1).
\end{split}
\end{equation*}

Next, we analyze the convergence rate for the cycle graph $C_n$. 
Again we assume the relation \eqref{p1} and let 
$\xi$ as in \eqref{p2}. Since $C_n$ has one more edge than $L_n$, we let $\eta \in \mathbb{R}$:  
\begin{equation*}
 \eta:=\frac{\Phi_{1}-\Phi_{n}}{\sqrt{n}\Delta x}=\sum_{i=1}^{n-1} \xi_i.
\end{equation*}
Substituting $\rho^{\infty}$ into \eqref{CE}, we have
\begin{equation}\label{16}
\begin{split}
\lambda_{\mathcal{H}}(\rho^{\infty})=&\min_{(\xi, \eta)\in\mathbb{R}^{n}_+} \{\frac{1}{\Delta x^2} 
[\xi^TA\xi+2\xi_1\eta+2\xi_{n-1}\eta+2\eta^2]\, : \\
\quad & \xi^T\xi+\eta^2=1, ~\eta=\sum_{i=1}^{n-1}\xi_i\}.
\end{split}
\end{equation}
The following transformations reduce \eqref{16} to a simpler eigenvalue problem.
Let
\begin{equation*}
\begin{pmatrix}\xi\\ \eta\end{pmatrix}=P\xi,\quad\textrm{where}\quad
P=\begin{pmatrix} 
I \\ \textbf{1} \end{pmatrix}\in \mathbb{R}^{n\times (n-1)}
\end{equation*}
with the identity matrix $I\in \mathbb{R}^{(n-1)\times(n-1)}$ and 
$\textbf{1}\in\mathbb{R}^{n-1}$ being the vector of all $1$'s.
Then, \eqref{16} becomes
\begin{equation}\label{p3}
\lambda_{\mathcal{H}}(\rho^{\infty})=
\min_{\xi\in \mathbb{R}^{n-1}_+}\{ \frac{1}{\Delta x^2}(P\xi)^T B (P\xi)~:~(P\xi)^T(P\xi)=1\},
\end{equation}
where
\begin{equation*}
B=\begin{pmatrix} 
  A & b^{\textrm{T}} \\ 
 b &  2
\end{pmatrix}\in \mathbb{R}^{n\times n} \quad \textrm{with}\quad b^T\in \mathbb{R}^{n-1},\
b=(1,0,\cdots, 0, 1),
\end{equation*}
and $A$ is as above.

Below, we compute \eqref{p3}.  First, we give explicit formulas
for the eigenvalues and eigenvectors of $B$.  

\begin{lemma}\label{Bspectrum}
Let $n\ge 3$.
For each $k=0,1,\dots, n-1$,  the eigenvalues of $B$ are
\begin{equation*}
\lambda_k=2- 2\cos(\frac{2k\pi}{n}).
\end{equation*}
For $k=0, 1,\cdots , n-1$, the associated
eigenvectors in un-normalized form are:
\begin{equation*}
v_k=(v_k(j))_{j=1}^{n}, \quad w_k=(w_k(j))_{j=1}^{n},
\end{equation*}
where, for $j=1,\cdots, n-1$,
\begin{equation*}
v_k(j)=\sin(\frac{2\pi k j}{n}) ,\quad w_k(j)=\cos(\frac{2\pi k j}{n});
\end{equation*}
and when $j=n$,
\begin{equation*}
v_k(n)=-\sin(\frac{2\pi k j}{n}) ,\quad w_k(j)=-\cos(\frac{2\pi k j}{n}).
\end{equation*}
\end{lemma}
\begin{proof}
The proof is by direct computation.  We just show the details for the case of $j=1$.
We have
\begin{equation*}
\begin{split}
(Bv_k)(1)=&2v_k(1)-v_{k}(2)+v_{k}(n)\\
=&2\sin(\frac{2\pi k }{n})-\sin(\frac{2\cdot 2\pi k}{n})-0\quad \textrm{By double angle formula}\\
=&(1-2\cos\frac{2k\pi}{n})v_k(1).
\end{split}
\end{equation*}
And 
\begin{equation*}
\begin{split}
(Bw_k)(1)=&2w_k(1)-w_{k}(2)+w_{k}(n)\\
=&2\cos(\frac{2\pi k }{n})-\cos(\frac{2\cdot 2\pi k}{n})+1\quad \textrm{By double angle formula}\\
=&(1-2\cos\frac{2k\pi}{n})w_k(1).
\end{split}
\end{equation*}
\end{proof}
Note that in Lemma \ref{Bspectrum}, many eigenvalues are repeated.  As a consequence, obviously
there are only two eigenvectors associated to each repeated eigenvalues, and not four; the repeating
eigenvalues, in fact, have identical pairs $v_k$, $w_k$, up to sign.  
However, the eigenvalue equal to $0$ is simple, with associated eigenvector
$w_0=(1,\cdots, 1, -1)^T$.  Moreover, aside from this $0$ eigenvalue, all other
eigenvalues are positive.

Now, observe that
$P^Tw_0=0$, and therefore the matrix $V=[w_0,P]$ is invertible and
$$BV=V\begin{bmatrix} 0 & 0 \\ 0 & C\end{bmatrix}, $$
where $C\in \mathbb{R}^{n-1,n-1}$.  Further, notice that $P^TP$ is positive
definite and thus it has a unique positive definite square root $(P^TP)^{1/2}$.
Thus, $\xi^TP^TBP\xi$, subject to $(P\xi)^TP\xi=1$, can be rewritten as
$$\xi^TP^TBP\xi=\xi^TP^TPC\xi=\xi^T(P^TP)^{1/2}(P^TP)^{1/2}C(P^TP)^{-1/2}(P^TP)^{1/2}\xi$$
and thus, with $x=(P^TP)^{1/2}\xi$, we end up with the problem
$$\min_{x: \ x^Tx=1} x^T\left[(P^TP)^{1/2}C(P^TP)^{-1/2}\right] x\,.$$
Finally, we notice that the matrix $\left[(P^TP)^{1/2}C(P^TP)^{-1/2}\right]$
is symmetric, and it is obviously similar to $C$, so that indeed 
\begin{equation}\label{21}\begin{split}
&\min_{x: \ x^Tx=1} x^T\left[(P^TP)^{1/2}C(P^TP)^{-1/2}\right] x = \\
&\min_{\xi\in \mathbb{R}^{n-1}_+}\{(P\xi)^T B (P\xi)~:~
(P\xi)^T(P\xi)=1\}=\textrm{The second smallest eigenvalue of $B$}.
\end{split}
\end{equation}

Putting it all together, \eqref{p3} gives
\begin{equation*}
\begin{split}
\lambda_{\mathcal{H}}(\rho^{\infty})=&\frac{1}{\Delta x^2}(\textrm{the second smallest eigenvalue of $B$})\\
=&\frac{1}{\frac{(b-a)^2}{(n-1)^2}}[2-2\cos(\frac{2\pi}{n})]=\frac{4\pi^2}{(b-a)^2}+o(1),
\end{split}
\end{equation*}
and the proof of Theorem \ref{mainClaim} is completed.
\end{proof}

\begin{thebibliography}{1}

\bibitem{ambrosio2003lecture}
L. Ambrosio.
\newblock {\em Lecture notes on optimal transport problems}.
\newblock Springer, 2003.

\bibitem{am2006}
L. Ambrosio, N. Gigli, and G. Savar{\'e}.
\newblock {\em Gradient flows: in metric spaces and in the space of probability
  measures}.
\newblock Springer Science \& Business Media, 2006.

\bibitem{bb}
JD. Benamou and Y. Brenier.
\newblock A computational fluid mechanics solution to the Monge-Kantorovich mass transfer problem.
\newblock{\em Numerische Mathematik} 84(3): 375--393, 2000.


\bibitem{benedetto1998non}
D. Benedetto, E. Caglioti, J. Carrillo and M. Pulvirenti.
\newblock A non-Maxwellian steady distribution for one-dimensional granular
  media.
\newblock {\em Journal of Statistical Physics}, 91(5-6):979--990, 1998.

\bibitem{upwind}
C. Buet and S. Cordier 
\newblock Numerical Analysis of Conservative and Entropy Schemes for the Fokker--Planck--Landau Equation.
\newblock{\em SIAM Journal on Numerical Analysis}, 953-973, 1999,

\bibitem{upwind1}
C. Buet and D. St\'ephane .
\newblock On the Chang and Cooper scheme applied to a linear Fokker-Planck equation.
\newblock{\em Communications in Mathematical Sciences} 8(4): 1079-1090, 2010.

\bibitem{CC}
J. Chang and G. Cooper.
\newblock{A practical difference scheme for Fokker-Planck equations.}
\newblock{Journal of Computational Physics} 6(1): 1-16, 1970.

\bibitem{c2015}
J. Carrillo, A. Chertock and Y. Huang.
\newblock A finite-volume method for nonlinear nonlocal equations with a gradient flow structure.
\newblock{\em Communications in Computational Physics}, volume 17, number 01, 233--258, 2015.

\bibitem{c2016}
J.A. Carrillo, Y. Huang, F.S. Patacchini and G. Wolansky.
\newblock Numerical Study of a Particle Method for Gradient Flows.
\newblock{\em arXiv:1512.03029}, 2015.

\bibitem{carrillo2003kinetic}
J. Carrillo, R. McCann and C. Villani.
\newblock Kinetic equilibration rates for granular media and related equations:
  entropy dissipation and mass transportation estimates.
\newblock {\em Revista Matematica Iberoamericana}, 19(3):971--1018, 2003.

\bibitem{JJM}
J. A. Carrillo, A. JuEngel, P. A. Markowich, G. Toscani and A. Unterreiter.
\newblock{Entropy Dissipation Methods for Degenerate Parabolic
Problems and Generalized Sobolev Inequalities}.
\newblock{Monatshefte f{\"u}r Mathematik}, 133(1): 1--82, 2001.

\bibitem{chow2012}
S.N. Chow, W. Huang, Y. Li and H. Zhou.
\newblock Fokker--Planck equations for a free energy functional or Markov process on a graph.
\newblock {\em Archive for Rational Mechanics and Analysis}, 203(3):969--1008,
  2012.


\bibitem{li2017}
S.N. Chow, W. Li and H. Zhou.
\newblock{Entropy dissipation of Fokker-Planck equations on graphs}
\newblock{\em arXiv:1701.04841}, 2017.

\bibitem{liVDP}
L. Dieci, W. Li and H. Zhou.
\newblock A new model for realistic random perturbations of stochastic oscillators. 
{\em Journal of Differential Equations}, 261(4): 2502--2527, 2016.

\bibitem{EM2}
M. Erbar and J. Maas.
\newblock Ricci curvature of finite Markov chains via convexity of the entropy.
\newblock{\em Archive for Rational Mechanics and Analysis} 206(3): 997--1038, 2012.

\bibitem{JKO}
R. Jordan, D. Kinderlehrer, and F. Otto.
\newblock{The variational formulation of the Fokker--Planck equation.}
\newblock{\em SIAM Journal on Mathematical Analysis}, 29(1) 1-17, 1998.

\bibitem{li_thesis}
W. Li.
\newblock A study of stochastic differential equations and Fokker-Planck equations with applications.
\newblock {\em PhD thesis}, Georgia tech, 2016.



\bibitem{EM1}
J. Maas.
\newblock{Gradient flows of the entropy for finite Markov chains.}
\newblock{\em Journal of Functional Analysis}, 261(8) 2250--2292, 2011.

\bibitem{EP}
P. A. Markowich and C. Villani.
\newblock On the trend to equilibrium for the Fokker-Planck equation: an interplay 
between physics and functional analysis.
\newblock{Mat. Contemp}, 19 1-29, 2000.

\bibitem{M}
A. Mielke.
\newblock A gradient structure for reaction--diffusion systems and for energy-drift-diffusion.
\newblock{\em Nonlinearity}, 24(4)13-29 2011.

\bibitem{O}
F. Otto.
\newblock The geometry of dissipative evolution equations: the porous medium equation.
\newblock{\em Communications in Partial Differential Equations}, Volume 26, Issue 1-2, 2001.

\bibitem{OV}
F. Otto, C{\'e}dric Villani.
\newblock Generalization of an inequality by Talagrand and links with the logarithmic Sobolev inequality.
\newblock{\em Journal of Functional Analysis,} 173.2 (2000): 361-400.


\bibitem{villani2002review}
C. Villani.
\newblock A review of mathematical topics in collisional kinetic theory.
\newblock {\em Handbook of mathematical fluid dynamics}, 1:71--305, 2002.


\bibitem{vil2003}
C. Villani.
\newblock {\em Topics in optimal transportation}.
\newblock Number~58. American Mathematical Soc., 2003.

\bibitem{vil2008}
C. Villani.
\newblock {\em Optimal transport: old and new}, volume 338.
\newblock Springer Science \& Business Media, 2008.

\end{thebibliography}
\end{document}